\documentclass{informs44}

\usepackage{amsmath,amssymb,amsfonts}
\usepackage{mathrsfs}
\usepackage[bookmarks=true,linkcolor=red,citecolor=blue,urlcolor=green,colorlinks=true, breaklinks]{hyperref}
\usepackage{enumerate}
\usepackage[textsize=footnotesize,textwidth=1.5cm]{todonotes}
\usepackage[ruled,vlined,procnumbered]{algorithm2e}
\usepackage{bm}
\usepackage{nicematrix}
\usepackage{booktabs}
\usepackage{threeparttable}
\usepackage{tikz}
\usepackage{bm}
\usepackage{pgfplots}
\usepackage{bookmark}
\usepackage{multirow}
\usepackage[normalem]{ulem}
\usepackage{longtable}
\usepackage{lscape}
\usepackage{makecell}
\usepackage{rotating}
\usepackage{fancyvrb}
\usepackage{anyfontsize} 
\usetikzlibrary{positioning}
\usepackage{comment}
\usepackage{tabularx}
\usepackage[margin=1in]{geometry}
\usepackage{graphicx}
\usepackage{array}
\usepackage{tcolorbox}
\usepackage{subcaption}
\usepackage{adjustbox}
\usepackage{pdflscape}
\usepackage{bm}

\usepackage{natbib}
\bibpunct[, ]{(}{)}{,}{a}{}{,}%
\def\bibfont{\small}%
\newtheorem{thm}{Theorem}

\newcommand{\BPname}{\texttt{BPCOL+}}
\newcommand{\BPUB}{\texttt{BPCOL+UB}}

\OneAndAHalfSpacedXI
\def\bibfont{\small}%
\pgfplotsset{compat=1.16}
\bibpunct[, ]{(}{)}{,}{a}{}{,}%
\graphicspath{{images/}}

\usepackage{etoolbox}
\newcommand{\zerodisplayskips}{%
  \setlength{\abovedisplayskip}{3pt}%
  \setlength{\belowdisplayskip}{3pt}%
  \setlength{\abovedisplayshortskip}{0pt}%
  \setlength{\belowdisplayshortskip}{0pt}}
\appto{\normalsize}{\zerodisplayskips}
\appto{\small}{\zerodisplayskips}
\appto{\footnotesize}{\zerodisplayskips}

\TheoremsNumberedThrough     
\ECRepeatTheorems
\JOURNAL{}
\EquationsNumberedThrough    

\begin{document}


\RUNAUTHOR{Zheng et al.}
\RUNTITLE{Branch-and-Price for Graph Coloring}


\TITLE{Advancing Branch-and-Price for Graph Coloring: New Pricing Strategies and Benchmark Results}

\ARTICLEAUTHORS{%
\AUTHOR{
Mingming Zheng\textsuperscript{a},
Roberto Baldacci\textsuperscript{b},
Fabio Furini\textsuperscript{c},
Qinghua Wu\textsuperscript{a}
}

\AFF{
\textsuperscript{a} School of Management, Huazhong University of Science and Technology, Wuhan 430074, China, 
\EMAIL{m.zheng0626@gmail.com}, \EMAIL{qinghuawu1005@gmail.com}\\
\textsuperscript{b} College of Science and Engineering, Hamad Bin Khalifa University, Doha 34110, Qatar, 
\EMAIL{rbaldacci@hbku.edu.qa}\\
\textsuperscript{c} Department of Computer, Control, and Management Engineering ``Antonio Ruberti'', Sapienza University of Rome, Rome 00185, Italy, 
\EMAIL{fabio.furini@uniroma1.it}
}
} 

\ABSTRACT{
This paper proposes \BPname, an exact branch-and-price algorithm for the Graph Coloring Problem. The algorithm is both novel and highly effective, integrating enhanced pricing strategies within Zero-Suppressed Binary Decision Diagrams (ZDDs) to efficiently solve the pricing problem associated with the maximal-stable-set-based set-covering formulation.
After computing upper and lower bounds at the root node using heuristic and column generation techniques, \BPname\ reduces the size of the ZDD through maximal stable set reduction techniques that exploit alternative dual vectors.
Computational experiments on the 137 DIMACS benchmark instances and on 5,000 recently proposed Erd\H{o}s--R\'enyi instances show that \BPname\ outperforms existing exact branch-and-price algorithms and remains highly competitive with state-of-the-art SAT-based exact solvers. In particular, \BPname\ solves 96 DIMACS instances within one hour and proves optimality for 4,641 of the 5,000 Erd\H{o}s--R\'enyi instances.
}%


\HISTORY{\today}

\maketitle

\section{Introduction}\label{sec:intro}

Given an undirected graph $\mathcal{G}:=(\mathcal{V},\mathcal{E})$, where $\mathcal{V}$ is the vertex set and $\mathcal{E}$ is the edge set, a \emph{proper vertex coloring} assigns a color to each vertex so that adjacent vertices receive different colors.
A \emph{stable set} is a subset $\mathcal{S}\subseteq\mathcal{V}$ whose vertices are pairwise nonadjacent; hence, $\mathcal{S}$ can be assigned a single color in a proper coloring. A \emph{maximal stable set} is a stable set that is inclusion-wise maximal. 
The \emph{Graph Coloring Problem} (GCP), also known as the \emph{Vertex Coloring Problem}, seeks a partition of $\mathcal{V}$ into the minimum number of stable sets. This value is called the \emph{chromatic number} of $\mathcal{G}$ and is denoted by $\chi(\mathcal{G})$.

As an illustrative example of the GCP, the left part of Figure~\ref{fig:example} depicts a graph $\mathcal{G}$ with $|\mathcal{V}|=6$ vertices and $|\mathcal{E}|=7$ edges.
The vertex set is $\mathcal{V}=\{v_1,v_2,v_3,v_4,v_5,v_6\}$, and the edge set consists of a $6$-cycle plus an additional edge between $v_1$ and $v_4$.
The right part of the figure shows a proper vertex coloring using two colors, which partitions $\mathcal{V}$ into the stable sets $\{v_1,v_3,v_5\}$ and $\{v_2,v_4,v_6\}$.
Since no proper vertex coloring with a single color exists, it follows that $\chi(\mathcal{G})=2$.
This example is used throughout the paper to illustrate key algorithmic components.

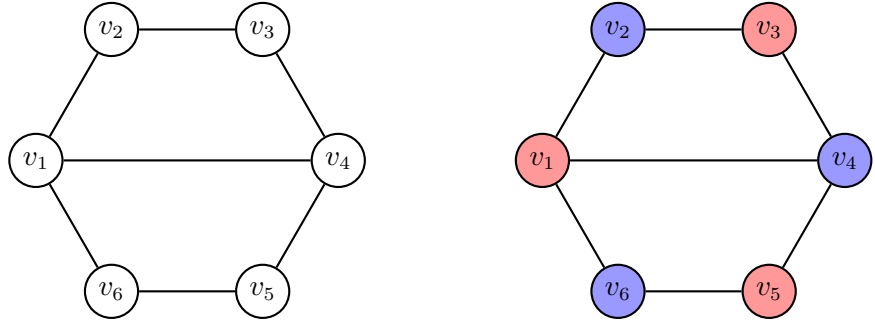
\begin{figure}[t]
\centering
\begin{subfigure}{.4\linewidth}
\centering
\begin{tikzpicture}[
  style=thick,
  scale=1,
  every node/.style={draw, circle, inner sep=1.5pt, minimum size=20pt, font=\normalsize, fill=white}
  ]
  \node (v1) at (180:2cm) {$v_1$};
  \node (v2) at (120:2cm) {$v_2$};
  \node (v3) at (60:2cm) {$v_3$};
  \node (v4) at (0:2cm) {$v_4$};
  \node (v5) at (-60:2cm) {$v_5$};
  \node (v6) at (-120:2cm) {$v_6$};
  
  \draw (v1)--(v2)--(v3)--(v4)--(v5)--(v6)--(v1);
  \draw (v1)--(v4);
\end{tikzpicture}
\end{subfigure}
\begin{subfigure}{.4\linewidth}
\centering
\begin{tikzpicture}[
  style=thick,
  scale=1,
  every node/.style={draw, circle, inner sep=1.5pt, minimum size=20pt, font=\normalsize}
  ]
  \node[fill=white!60!red] (v1) at (180:2cm) {$v_1$};
  \node[fill=white!60!blue] (v2) at (120:2cm) {$v_2$};
  \node[fill=white!60!red] (v3) at (60:2cm) {$v_3$};
  \node[fill=white!60!blue] (v4) at (0:2cm) {$v_4$};
  \node[fill=white!60!red] (v5) at (-60:2cm) {$v_5$};
  \node[fill=white!60!blue] (v6) at (-120:2cm) {$v_6$};
  
  \draw (v1)--(v2)--(v3)--(v4)--(v5)--(v6)--(v1);
  \draw (v1)--(v4);
\end{tikzpicture}
\end{subfigure}
\vspace{0.3cm}
\caption{An example graph and a proper vertex coloring using two colors.}
\label{fig:example}
\end{figure}

The GCP is a classical and computationally challenging $\mathcal{NP}$-hard combinatorial optimization problem in graph theory~\citep{Garey1979}.
Beyond its theoretical relevance, it has numerous applications, including scheduling, register allocation in compilers, and frequency assignment in wireless communication systems \citep[see, e.g.,][]{pardalos1998graph}.
For a comprehensive survey, we refer the interested reader to~\citet{malaguti2010review}.

\subsection{Related Work on Exact Algorithms for the GCP}\label{sec:lit}

The GCP has been extensively studied over the past decades, and numerous exact and heuristic algorithms have been proposed in the literature. Given the focus of this paper on exact methods, this section provides a brief overview of exact algorithms for the GCP. To the best of our knowledge, the most effective exact approaches can be broadly classified into two main research streams: Branch-and-Price (BP) algorithms and satisfiability (SAT)-based algorithms.

Significant progress has been made in developing exact BP algorithms for the GCP since the seminal work of \citet{Mehrotra1996}. \citet{malaguti2011exact} proposed an approach combining an effective heuristic for computing tight upper bounds with advanced pricing methods and two alternative branching schemes. \citet{Held2012} improved the Column Generation (CG) bounding phase by introducing a method to compute safe lower bounds. \citet{Gualandi2012} developed a hybrid method that integrates constraint programming into the column-generation framework to solve the Pricing Problem (PP) more efficiently and proposed an advanced approach to explore the branching tree. \citet{morrison2014characteristics} designed advanced branching rules that preserve the PP structure throughout the branching tree.  More recently, \textit{Decision Diagrams} (DDs) have emerged as a powerful enhancement to BP algorithms for the GCP. \citet{Morrison2016} incorporated \emph{Zero-Suppressed Binary Decision Diagrams} (ZDDs) into a BP framework to efficiently represent and manage the set of maximal stable sets in the PP. 

Machine learning has also been explored to improve CG and BP for the GCP by providing learned guidance on column selection and pricing. In earlier work, \citet{shen2022enhancing} trained an offline linear support vector machine to predict structural properties of optimal solutions to the PP and used these predictions to guide sampling that generates many high-quality stable sets with negative reduced cost in one shot, accelerating the solution of the linear programming (LP) relaxation within CG. Later, \citet{sun2022learning} shifted the learning target from negative reduced cost to the likelihood that a column appears in an optimal integer coloring, and used the predictor to iteratively sample and filter maximal stable sets, producing a diverse, high-quality column pool for constructing primal solutions and for seeding CG. These works report mainly process-level BP/column-generation metrics for instances from the literature.

SAT-based approaches have also emerged as a powerful alternative for solving the GCP, supported by the rapid advances in modern SAT solvers. These algorithms can flexibly incorporate customized branching strategies, often achieving strong competitiveness compared to BP methods. \citet{hebrard2020constraint} proposed a hybrid Constraint Programming and SAT algorithm based on Zykov's branching scheme, introduced Mycielskian-based lower bounds, and specialized pruning techniques to enhance efficiency. \citet{heule2022cliques} developed an exact SAT-based algorithm that alternates between identifying cliques of larger size and colorings with fewer colors to improve the performance of the algorithm. \citet{faber2024sat} introduced a partial-ordering SAT encoding for graph coloring, also extended to handle bandwidth constraints. More recently, \citet{brand2026customized} proposed a Zykov tree-inspired SAT algorithm that combines transitivity constraints and incremental search to improve propagation and solver performance. This recent algorithm is the current state-of-the-art SAT-based algorithm for the GCP.

Recently, \citet{van2022graph} introduced a DD-based network-flow formulation, yielding another exact method for the GCP, although it is not yet competitive with BP and SAT-based algorithms. In addition, several \texttt{DSATUR}-based branch-and-bound algorithms have been proposed; see, e.g., \citet{sewell1996improved}, \citet{san2012new}, and \citet{furini2017improved}.

Despite the significant progress in the literature, focusing on BP remains crucial because it directly addresses two persistent bottlenecks. First, many instances exhibit a very small optimality gap, yet existing BP methods often struggle to efficiently certify optimality; incorporating dual-information-based reduction techniques can substantially speed up the solution of the PP, thereby accelerating the solution of these ``near-optimal but hard-to-prove'' cases. Second, ZDD-based BP is often limited on large-scale or sparse instances due to the difficulty of fully enumerating stable sets, even though a portion of such instances can be proven optimal already at the root node; this calls for a more comprehensive BP framework that fully exploits the strengths of ZDDs when they are effective, while relying on a competitive pricing mechanism to certify optimality robustly when ZDD construction is not viable, thus covering a broader range of instances.

\subsection{Contributions and Outline of the Paper}\label{sec:contri}

This paper presents \BPname, a ZDD-based BP algorithm for the GCP. \BPname\ carefully integrates state-of-the-art techniques, including numerically safe lower-bound computation, problem reductions, and primal-bound computation methods, with the following key contributions:
\begin{itemize}
    \item \BPname\ introduces novel maximal-stable-set-reduction techniques into the ZDD construction. 
    Leveraging alternative dual vectors, \BPname\ builds a compact reduced ZDD that excludes maximal stable sets certified as irrelevant for the target value, thereby improving pricing efficiency.

    \item \BPname\ is evaluated through a comprehensive computational study on the 137 DIMACS benchmark instances and on recently proposed 5,000 Erd\H{o}s--R\'enyi graph instances. All algorithms with publicly available source code are tested under the same computational environment to ensure a fair comparison. The results show that \BPname\ significantly outperforms existing exact BP algorithms in terms of computational efficiency while remaining highly competitive with state-of-the-art SAT-based exact solvers. \BPname\ can solve 96 DIMACS instances within one hour, as well as several previously open Erd\H{o}s--R\'enyi instances.
    
    \item To foster transparency, reproducibility, and future comparisons, we provide a public repository containing all benchmark instances used in our experiments, the corresponding computational results, the complete source code, and detailed implementation information.
\end{itemize}

The remainder of this paper is organized as follows. Section \ref{sec:SC} presents the set-covering formulation of the GCP and describes the computation of the lower bound, as well as the branching strategy embedded in \BPname. Section \ref{sec:pricalg-prics} details the design of the pricing algorithm based on ZDDs and the corresponding enhanced pricing strategy. An overview of \BPname\ is provided in Section \ref{sec:algoverview}, followed by a computational study in Section \ref{sec:compstud} evaluating the performance of \BPname\ and the effectiveness of its key components. Finally, we conclude the paper in Section \ref{sec:conclusions}. 
Additional material, including the proofs of formal statements, ZDD construction details, computational results, and experiments on the main algorithmic components are provided in the e-companion.

\section{Formulation, Lower Bound Computation and Branching Scheme}\label{sec:SC}

We introduce an \textit{Integer Linear Programming} (ILP) formulation of the GCP with an exponential number of variables, originally proposed by \citet{Mehrotra1996} and known as the \textit{set-covering} formulation.

Let $\mathscr{S}$ denote the family of all maximal stable sets of $\mathcal{G}$, and for each vertex $v\in\mathcal{V}$ let $\mathscr{S}(v)\subseteq\mathscr{S}$ be the subfamily of maximal stable sets containing $v$.
For each $\mathcal{S}\in\mathscr{S}$, we introduce a binary variable $\xi_{\mathcal{S}}$ equal to~1 if and only if $\mathcal{S}$ is selected in the solution; in this case, the vertices in $\mathcal{S}$ take the same color.
The set-covering formulation reads as follows:
\begin{equation}
\label{SC}
\min_{{\boldsymbol \xi} \in \{0,1\}^{\mathscr{S}}} \left\{ \sum_{\mathcal{S} \in \mathscr{S}} \xi_{\mathcal{S}}:~~  
\sum_{\mathcal{S} \in \mathscr{S}(v)} \xi_{\mathcal{S}} \ge 1,~ v \in \mathcal{V} \right\}.
\end{equation}

The objective function in~\eqref{SC} minimizes the number of selected maximal stable sets.
The set-covering constraints in~\eqref{SC} ensure that every vertex is covered at least once.
Any optimal integer solution of~\eqref{SC} corresponds to a minimum-cardinality collection of maximal stable sets covering $\mathcal{V}$, and its value equals the chromatic number $\chi(\mathcal{G})$.
In an optimal solution of~\eqref{SC}, vertices may be covered multiple times; a proper vertex coloring can always be obtained by assigning each vertex to exactly one selected stable set.

Replacing the binary variables in~\eqref{SC} with non-negative variables yields the \textit{Linear Programming} (LP) relaxation, commonly referred to as the \textit{Fractional Graph Coloring Problem}.
Let $\overline{\chi}(\mathcal{G})$ denote the optimal value of this LP relaxation, called the \textit{fractional chromatic number}, which typically provides a tight lower bound on $\chi(\mathcal{G})$.
Since formulation~\eqref{SC} features an exponential number of variables, \BPname\ solves it via a specialized BP algorithm. For general background on column generation and the design of BP solution approaches, we refer the reader to the books of \citet{uchoa2024optimizing} and \citet{desrosiers2024branch}.

The remainder of this section discusses the computation of the lower bound and the adopted branching scheme (\S\ref{subsec:ColumnGenerationBranchingScheme}), followed by the formulation of the PP associated with the set-covering formulation (\S\ref{subsec:pricing}).

\subsection{Column Generation and Branching Scheme} \label{subsec:ColumnGenerationBranchingScheme}

At each node of the enumeration tree, \BPname\ tackles the LP relaxation of \eqref{SC}, the \textit{master problem} (MP), via CG, generating columns by repeatedly solving the PP. At each CG iteration, the \textit{restricted master problem} (RMP), i.e., the MP restricted to the columns generated so far, is solved with the simplex algorithm to obtain primal and dual solutions; then, the PP is solved to identify columns with negative reduced cost.
If such columns are found, they are added to the RMP and a new CG iteration is performed; otherwise, CG terminates, and the current RMP solution is optimal for the LP relaxation at the branching node.
We refer the interested reader to~\citet{BarnhartJNSV98,LubbeckeD05} for further details on CG.

\BPname\ adopts the branching scheme commonly known as \textit{branching on master variables}, which selects a fractional variable and creates two child branching nodes by fixing the selected variable to~0 or~1.
This rule has been successfully employed in BP algorithms for the GCP by~\citet{malaguti2011exact} and~\citet{Morrison2016}.
A known drawback is that it disrupts the structure of the PP~\citep{morrison2014characteristics, Morrison2016}, because fixing a variable to~0 requires the PP to prevent regenerating the same variable at descendant branching nodes.
In \BPname, this is handled efficiently via ZDD-based operations (see \S\ref{subsec:branching}). Below, we first discuss the impact of the adopted branching scheme on the RMP and then formulate the resulting PP. 

Consider the RMP at a generic node of the enumeration tree, and let $\overline{\mathscr{S}} \subseteq \mathscr{S}$ denote the set of maximal stable sets corresponding to the variables of the RMP. Let $\overline{\mathscr{S}}_0\subseteq\overline{\mathscr{S}}$ and $\overline{\mathscr{S}}_1\subseteq\overline{\mathscr{S}}$ denote the subsets of maximal stable sets whose associated variables are fixed to~0 and~1, respectively. These branching decisions are enforced directly in the RMP.
Maximal stable sets associated with variables fixed to~1 already cover a subset of vertices; accordingly, we define the sets of \textit{covered} and \textit{uncovered} vertices as
\begin{equation}
\mathcal{V}_c \;:=\; \bigcup_{\mathcal{S}\in\overline{\mathscr{S}}_1}\mathcal{S}
\quad \text{and} \quad
\mathcal{V}_u  \;:=\; \mathcal{V} \setminus \mathcal{V}_c .
\label{covered}
\end{equation}
Fixing a variable $\xi_{\mathcal{S}}$ to $1$ selects the maximal stable set $\mathcal{S}$. In the current RMP, this is implemented by setting the lower bound of the corresponding variable to $1$. During backtracking, when returning to a parent branching node, the original lower bound is restored.
Fixing a variable $\xi_{\mathcal{S}}$ to $0$ is implemented by setting its upper bound to $0$ in the RMP; the corresponding maximal stable set is then implicitly forbidden when solving the PP. This upper bound is removed during backtracking. This implementation preserves a consistent RMP structure throughout the algorithm, enabling the LP solver to warm start using previously computed simplex information (e.g., bases) across consecutive CG iterations and branching nodes.

\subsection{Pricing Problem}\label{subsec:pricing}

Let ${\boldsymbol \pi}\in\mathbb{R}^{|\mathcal{V}|}_{\ge 0}$ be a dual solution of the RMP, where $\pi_v$ is the dual variable associated with the set-covering constraint in~\eqref{SC} associated with vertex $v$. The reduced cost of the variable associated with a maximal stable set $\mathcal{S}\in\mathscr{S}$ computed with respect to the dual solution ${\boldsymbol \pi}$ is
\begin{equation}
\label{RC}
rc(\mathcal{S},{\boldsymbol \pi}) \;:=\; 1-\sum_{v\in\mathcal{S}}\pi_v,
\end{equation}
which also corresponds to the slack of the corresponding dual constraint.
At the current branching node, variables corresponding to sets in $\overline{\mathscr{S}}_0$ are forbidden by branching. Therefore, the PP amounts to finding a non-forbidden maximal stable set of minimum reduced cost, and is formulated as follows:
\begin{equation}
\label{eq:pricing}
\widetilde{\mathcal{S}} \in
\arg\min_{\mathcal{S}\in\mathscr{S}} \left\{~rc(\mathcal{S},{\boldsymbol \pi}) :~ \mathcal{S}\notin\overline{\mathscr{S}}_0 ~\right\}.
\end{equation}
Let $rc^*({\boldsymbol \pi}) := rc(\widetilde{\mathcal{S}}, {\boldsymbol \pi})$ denote the minimum reduced cost.
If $rc^*({\boldsymbol \pi})<0$, the corresponding variable $\xi_{\widetilde{\mathcal{S}}}$ is added to the RMP; otherwise, no improving variable exists, and CG can terminate. 
At the root node, where $\overline{\mathscr{S}}_0=\emptyset$, the PP reduces to a \textit{maximum weight stable set problem} (MWSSP) with vertex weights given by the dual variables. \BPname\ exploits exact and heuristic MWSSP techniques (see \S\ref{subsec:lb_zdd}). At branching nodes, it uses ZDD-based operations to efficiently solve the PP, as described in the next section.

\section{An Enhanced ZDD-Based Pricing Algorithm}\label{sec:pricalg-prics}

This section describes the pricing algorithm used to solve the PP \eqref{eq:pricing}. Section \ref{sec:prica-ZDD} briefly illustrates the construction of a ZDD and its use in solving the PP. Section \ref{sec:prica-ZDD-enh} then presents techniques to reduce the size of a ZDD, thereby accelerating the column generation solution process. 

\subsection{Zero-Suppressed Decision Diagrams for Solving the Pricing Problem}\label{sec:prica-ZDD}

The ZDD is an extension of the Binary Decision Diagram (BDD). The use of ZDDs to solve the PP offers significant advantages within the BP framework for the GCP \citep{Morrison2016}. In particular, ZDDs provide a compact representation of the set of maximal stable sets, thereby capturing the entire solution space of the PP in a graph-based structure.
Given the graph $\mathcal{G}$, we fix an ordering of the vertices (see \S\ref{subsec:preprocessing}), denoted by $\mathcal{V}:=\{v_1,v_2,\ldots,v_n\}$. The chosen ordering can significantly affect the size of the resulting ZDD. A ZDD can then be constructed to represent the solution space of the PP on $\mathcal{G}$. This ZDD is a layered directed acyclic graph $\mathcal{Z}:=(\mathcal{N},\mathcal{A})$ with node set $\mathcal{N}$ and arc set $\mathcal{A}$. The size of a ZDD is measured as \(|\mathcal N|+|\mathcal A|\), where \(\mathcal N\) includes all decision nodes and the two terminal nodes, and \(\mathcal A\) is the set of arcs. Graph $\mathcal{Z}$ comprises $n+1$ layers, where layer $j$ consists of nodes associated with vertex $v_j \in \mathcal{V}$, for $j=1,\ldots,n$, and consists of the following components (see Figure \ref{fig:comparison_zdd} for an example with $|\mathcal{V}|=6$ and $|\mathcal{E}|=7$):
\begin{enumerate}
    \item \textbf{Root and terminal nodes.}  
   $\mathcal{Z}$ contains a unique \emph{root} node, appearing first in the topological order of the directed acyclic graph (node~$r$ in Figure \ref{fig:comparison_zdd}(b)).  
    It also includes two terminal nodes: the \textsc{True} and \textsc{False} nodes, represented by $1$ and $0$, respectively (square nodes in Figure \ref{fig:comparison_zdd}(b)).

    \item \textbf{Non-terminal nodes and outgoing arcs.}  
    Each non-terminal node $a \in \mathcal{N} \setminus \{0,1\}$ is associated with a vertex in $\mathcal{V}$.  
    Let $var(a)$ denote the index of the vertex associated with node $a$, i.e., $var(a)=i$ if $a$ corresponds to vertex $v_i \in \mathcal{V}$.  
    For example, in Figure \ref{fig:comparison_zdd}(b), node~$d$ corresponds to vertex $v_4$, thus $var(d)=4$.  
    Node $d$ has two outgoing arcs: the \emph{high} arc and the \emph{low} arc.  
    The high arc leads to its \emph{high child node} $hi(d)$, while the low arc leads to its \emph{low child node} $lo(d)$.  
    In Figure \ref{fig:comparison_zdd}(b), the solid arc from node~$d$ to node~$c$ represents the high arc ($hi(d)=c$), and the dashed arc from node~$d$ to node~$a$ represents the low arc ($lo(d)=a$).  
    The indices of the vertices associated with the children are strictly larger than that of the parent node; that is, if $var(d)=i$, then $var(hi(d))>i$ and $var(lo(d))>i$. For example, $var(hi(d))=var(c)=6>4=var(d)$ and $var(lo(d))=var(a)=5>4=var(d)$.

    \item \textbf{Zero-suppression property.}  
For any non-terminal node \(a \in \mathcal{N} \setminus \{0,1\}\), its high child cannot be the \textsc{False} node; that is, \(hi(a)\neq 0\).  
This zero-suppression property is the key distinction between ZDDs and conventional BDDs.
\end{enumerate}

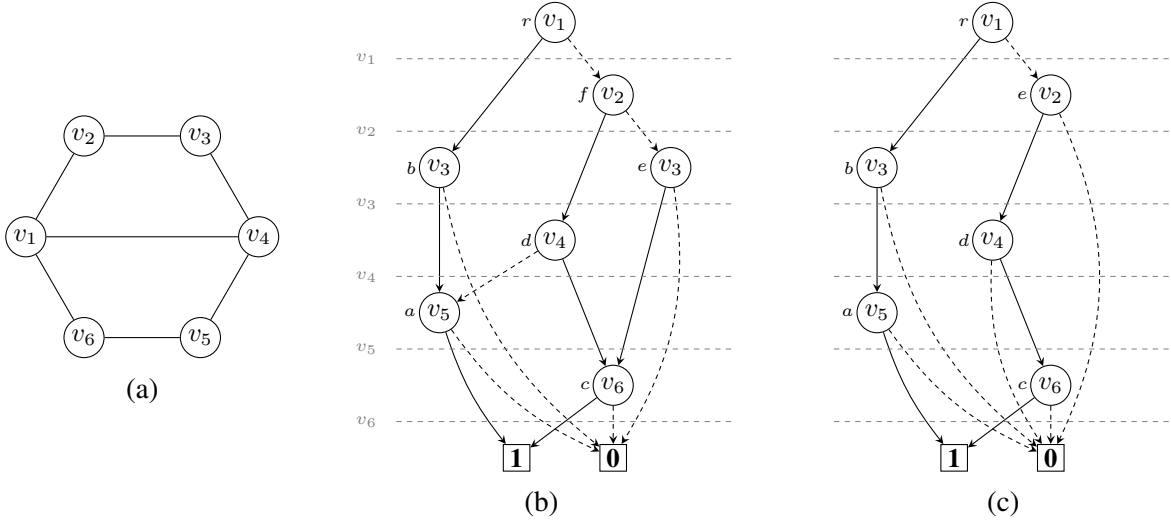
\begin{figure}[t]
  \centering
  \begin{subfigure}[c]{0.26\textwidth}
    \centering
    \begin{tikzpicture}[
			scale=0.7,
			baseline=(current bounding box.center),
			every node/.style={draw, circle, inner sep=1.5pt, minimum size=12pt, font=\small}
			]
			\node (v2) at ( 120:2.2cm) {$v_2$};
			\node (v3) at ( 60:2.2cm) {$v_3$};
			\node (v4) at (  0:2.2cm) {$v_4$};
			\node (v5) at (-60:2.2cm) {$v_5$};
			\node (v6) at (-120:2.2cm) {$v_6$};
			\node (v1) at (180:2.2cm) {$v_1$};
			\draw (v2)--(v3)--(v4)--(v5)--(v6)--(v1)--(v2);
			\draw (v1)--(v4);
		\end{tikzpicture}
    \caption{}
  \end{subfigure}\hfill%
  \begin{subfigure}[c]{0.36\textwidth}
    \centering
    \begin{tikzpicture}[
		scale=0.6,
		>=stealth,
		circle node/.style={circle, draw, inner sep=1.5pt, minimum size=10pt, font=\small},
		square node/.style={rectangle, draw, inner sep=1.5pt, minimum size=10pt, font=\small},
		one arc/.style={->},
		zero arc/.style={->, dashed, dash pattern=on 1.8pt off 1.6pt},
		baseline=(current bounding box.center)]
		\node[circle node] (v1) at (4.0,0) {$v_1$};
		\node[circle node] (v2) at (5.28,-1.6) {$v_2$};
		\node[circle node] (v3_left) at (1.45,-3.2) {$v_3$};
		\node[circle node] (v3_right) at (6.55,-3.2) {$v_3$};
		\node[circle node] (v4_left) at (4.0,-4.8) {$v_4$};
		\node[circle node] (v5) at (1.45,-6.4) {$v_5$};
		\node[circle node] (v6) at (5.28,-8.0) {$v_6$};
		\node[square node] (one)  at (3.15,-9.6) {\bfseries 1};
		\node[square node] (zero) at (5.28,-9.6) {\bfseries 0};
		
		\draw[one arc] (v1) -- (v3_left);
		\draw[zero arc] (v1) -- (v2);
		\draw[one arc] (v2) -- (v4_left);
		\draw[zero arc] (v2) -- (v3_right);
		
		\draw[one arc] (v3_left) -- (v5);
		\draw[zero arc] (v3_left) to[bend right=20] (zero);
		\draw[one arc] (v3_right) -- (v6);
		\draw[zero arc] (v3_right) to[bend left=20] (zero);
		\draw[one arc] (v4_left) -- (v6);
		\draw[zero arc] (v4_left) -- (v5);
		\draw[one arc] (v5) to[bend right=10] (one);
		\draw[zero arc] (v5) to[bend right=15] (zero);
		\draw[one arc] (v6) -- (one);
		\draw[zero arc] (v6) -- (zero);
		
		\node[left, xshift=-5pt, font=\tiny] at (v1) {$r$};
		\node[left, xshift=-5pt, font=\tiny] at (v2) {$f$};
		\node[left, xshift=-5pt, font=\tiny] at (v3_left) {$b$};
		\node[left, xshift=-5pt, font=\tiny] at (v3_right) {$e$};
		\node[left, xshift=-5pt, font=\tiny] at (v4_left) {$d$};
		\node[left, xshift=-5pt, font=\tiny] at (v5) {$a$};
		\node[left, xshift=-5pt, font=\tiny] at (v6) {$c$};
		
		\draw[gray, dashed, dash pattern=on 2pt off 2pt, line width=0.35pt] (0.5,-0.8) -- (8.0,-0.8);
		\node[left, font=\tiny, gray] at (0.3,-0.8) {$v_1$};
		
		\draw[gray, dashed, dash pattern=on 2pt off 2pt, line width=0.35pt] (0.5,-2.4) -- (8.0,-2.4);
		\node[left, font=\tiny, gray] at (0.3,-2.4) {$v_2$};
		
		\draw[gray, dashed, dash pattern=on 2pt off 2pt, line width=0.35pt] (0.5,-4.0) -- (8.0,-4.0);
		\node[left, font=\tiny, gray] at (0.3,-4.0) {$v_3$};
		
		\draw[gray, dashed, dash pattern=on 2pt off 2pt, line width=0.35pt] (0.5,-5.6) -- (8.0,-5.6);
		\node[left, font=\tiny, gray] at (0.3,-5.6) {$v_4$};
		
		\draw[gray, dashed, dash pattern=on 2pt off 2pt, line width=0.35pt] (0.5,-7.2) -- (8.0,-7.2);
		\node[left, font=\tiny, gray] at (0.3,-7.2) {$v_5$};
		
		\draw[gray, dashed, dash pattern=on 2pt off 2pt, line width=0.35pt] (0.5,-8.8) -- (8.0,-8.8);
		\node[left, font=\tiny, gray] at (0.3,-8.8) {$v_6$};
	\end{tikzpicture}
    \caption{}
  \end{subfigure}\hfill%
  \begin{subfigure}[c]{0.36\textwidth}
    \centering
    \begin{tikzpicture}[
		scale=0.6,
		>=stealth,
		circle node/.style={circle, draw, inner sep=1.5pt, minimum size=10pt, font=\small},
		square node/.style={rectangle, draw, inner sep=1.5pt, minimum size=10pt, font=\small},
		one arc/.style={->},
		zero arc/.style={->, dashed, dash pattern=on 1.8pt off 1.6pt},
		baseline=(current bounding box.center)]
		\node[circle node] (v1) at (4.0,0) {$v_1$};
		\node[circle node] (v2) at (5.28,-1.6) {$v_2$};
		\node[circle node] (v3_left) at (1.45,-3.2) {$v_3$};
		\node[circle node] (v4_left) at (4.0,-4.8) {$v_4$};
		\node[circle node] (v5) at (1.45,-6.4) {$v_5$};
		\node[circle node] (v6) at (5.28,-8.0) {$v_6$};
		\node[square node] (one)  at (3.15,-9.6) {\bfseries 1};
		\node[square node] (zero) at (5.28,-9.6) {\bfseries 0};
		
		\draw[one arc] (v1) -- (v3_left);
		\draw[zero arc] (v1) -- (v2);
		\draw[one arc] (v2) -- (v4_left);
		\draw[zero arc] (v2) to[bend left=25] (zero);
		
		\draw[zero arc] (v3_left) to[bend right=20] (zero);
		\draw[one arc] (v3_left) -- (v5);
		
		\draw[one arc] (v4_left) to (v6);
		\draw[zero arc] (v4_left) to[bend right=18] (zero);
		\draw[one arc] (v5) to[bend right=10] (one);
		\draw[zero arc] (v5) to[bend right=15] (zero);
		\draw[one arc] (v6) -- (one);
		\draw[zero arc] (v6) -- (zero);
		
		\node[left, xshift=-5pt, font=\tiny] at (v1) {$r$};
		\node[left, xshift=-5pt, font=\tiny] at (v2) {$e$};
		\node[left, xshift=-5pt, font=\tiny] at (v3_left) {$b$};
		\node[left, xshift=-5pt, font=\tiny] at (v4_left) {$d$};
		\node[left, xshift=-5pt, font=\tiny] at (v5) {$a$};
		\node[left, xshift=-5pt, font=\tiny] at (v6) {$c$};
		
		\draw[gray, dashed, dash pattern=on 2pt off 2pt, line width=0.35pt] (0.5,-0.8) -- (8.0,-0.8);
		\draw[gray, dashed, dash pattern=on 2pt off 2pt, line width=0.35pt] (0.5,-2.4) -- (8.0,-2.4);
		\draw[gray, dashed, dash pattern=on 2pt off 2pt, line width=0.35pt] (0.5,-4.0) -- (8.0,-4.0);
		\draw[gray, dashed, dash pattern=on 2pt off 2pt, line width=0.35pt] (0.5,-5.6) -- (8.0,-5.6);
		\draw[gray, dashed, dash pattern=on 2pt off 2pt, line width=0.35pt] (0.5,-7.2) -- (8.0,-7.2);
		\draw[gray, dashed, dash pattern=on 2pt off 2pt, line width=0.35pt] (0.5,-8.8) -- (8.0,-8.8);
	\end{tikzpicture}
    \caption{}
  \end{subfigure}\vspace{0.5cm}

  \caption{Example of construction of a ZDD and associated reduced ZDD: (a) Graph $\mathcal{G}=(\mathcal{V},\mathcal{E})$, (b) ZDD graph $\mathcal{Z}=(\mathcal{N},\mathcal{A})$ and (c) ZDD reduced graph.}
  \label{fig:comparison_zdd}
\end{figure}

Any path \(\mathcal{P}_{\mathcal S}\) from the root node to a terminal node in \(\mathcal{Z}\) uniquely corresponds to a subset of vertices \(\mathcal S \subseteq \mathcal{V}\). Starting from the root node, if the next node along the path is the high child \(hi(a)\), then \(v_{var(a)} \in \mathcal S\); otherwise, if the path follows the low child \(lo(a)\), then \(v_{var(a)} \notin \mathcal S\). The terminal node reached by \(\mathcal{P}_{\mathcal S}\) represents the output of \(\mathcal S\) on \(\mathcal{Z}\), denoted by \(\mathcal{Z}(\mathcal S)\). The output \(\mathcal{Z}(\mathcal S)\) takes only two values: \(\mathcal{Z}(\mathcal S)=1\) indicates that \(\mathcal{Z}\) \emph{accepts} \(\mathcal S\), whereas \(\mathcal{Z}(\mathcal S)=0\) means that \(\mathcal{Z}\) \emph{rejects} \(\mathcal S\).
\(\mathcal{Z}\) is reduced by applying the standard node-merging rule: any two nodes in \(\mathcal{Z}\) that correspond to the same vertex in \(\mathcal{G}\) and have identical low and high children are merged into a single node, thereby reducing the total number of nodes in \(\mathcal{Z}\).

The example in Figure~\ref{fig:comparison_zdd}(a) shows that the graph has four maximal stable sets:
\(S_1=\{v_1,v_3,v_5\}\), \(S_2=\{v_2,v_4,v_6\}\), \(S_3=\{v_2,v_5\}\), and \(S_4=\{v_3,v_6\}\).
Figure~\ref{fig:comparison_zdd}(b) shows the complete ZDD, which contains seven decision nodes (labeled \(r,f,b,e,d,a,c\)) and encodes all four maximal stable sets through distinct root-to-\textsc{True} paths.
For instance, consider the path \(\{r,f,d,c,1\}\).
Since \(r\) and \(f\) are connected by a dashed arc (the low arc), the corresponding vertex \(v_1\) is excluded from the set.
In contrast, \(f\) and \(d\) are connected by a solid arc (the high arc), so \(v_2\) is included.
Similarly, the subsequent solid arcs imply that \(v_4\) and \(v_6\) are included as well.
Hence, the path \(\{r,f,d,c,1\}\) corresponds to the maximal stable set \(S_2=\{v_2,v_4,v_6\}\). Since the path ends at the \textsc{True} node, we have \(\mathcal{Z}(S_2)=1\), meaning that \(\mathcal{Z}\) accepts \(S_2\).

\subsubsection{Solving the Pricing Problem.}

Given the family \(\mathscr{S}\) of all maximal stable sets of graph \(\mathcal{G}\), let \(\mathcal{Z}_{\mathscr{S}}\) denote the ZDD encoding \(\mathscr{S}\). For a given instance of the PP \eqref{eq:pricing} with dual vector \({\boldsymbol \pi}=(\pi_v)_{v\in\mathcal{V}}\), each non-terminal node \(a\in\mathcal{N}\setminus\{0,1\}\), labeled by vertex \(v_{\mathrm{var}(a)}\), is assigned weight \(\pi_{v_{\mathrm{var}(a)}}\) on its high arc and weight \(0\) on its low arc.
Thus, the PP reduces to finding a longest path from the root to the \textsc{True} node in the weighted graph $\mathcal{Z}_{\mathscr{S}}$, which can be efficiently solved by topological sorting and dynamic programming \citep{Morrison2016}. The resulting path corresponds to the maximal stable set $\widetilde{\mathcal{S}} \in \mathscr{S}$ (see \eqref{eq:pricing}) whose associated variable has minimum reduced cost $rc^*({\boldsymbol \pi})$. 

After each pricing iteration, once the column associated with a maximal stable set $\mathcal S$ is added to the RMP, the ZDD is updated to avoid regenerating $\mathcal S$.  We adopt the \textit{RestrictSet} procedure of \citet{Morrison2016}, which removes exactly the path corresponding to $\mathcal S$ from the diagram $\mathcal{Z}_{\mathscr{S}}$ while preserving all other solutions, so that the updated ZDD directly represents the constrained PP at the next iteration. Specifically, the algorithm identifies the unique root-to-\textsc{True} path $P_{\mathcal S}$ and redirects its endpoint to the \textsc{False} node. If $P_{\mathcal S}$ shares a suffix with other solutions, \textit{RestrictSet} duplicates the overlapping portion from the first node with indegree greater than one (the split node) and redirects only the duplicate to the \textsc{False} node, thereby excluding only $\mathcal S$ while preserving all other feasible paths.  Similarly, branching decisions encoded in the set $\overline{\mathscr{S}}_0$ are handled by removing the corresponding maximal stable sets from $\mathcal{Z}_{\mathscr{S}}$. For additional details, see Algorithm~1 in \citet{Morrison2016}.

\subsection{Reduced-Cost Fixing and ZDD-Based Column Reduction}\label{sec:prica-ZDD-enh}

A classical result in integer programming, often referred to as \emph{reduced-cost fixing} and dating back to the seminal paper on the Traveling Salesman Problem by \citet{dantzig1954solution}, states that, given an upper bound $UB$ and the optimal value $LB$ of the LP relaxation of a binary integer program, any variable with reduced cost larger than the \textit{gap} $=UB-LB$ cannot take a positive value in an optimal integer solution. Reduced-cost fixing is a popular technique for accelerating the solution of \textit{compact} mixed-integer linear programs (MILPs).
For integer programming formulations with a huge number of variables, referred to as \textit{extensive} formulations and typically solved by a BP algorithm, reduced-cost fixing is not straightforward to apply directly to master-problem variables, because variables fixed to zero must also be prevented from being regenerated by the pricing subproblem. However, two main types of techniques have been used in the literature: (i) reduced-cost fixing within the pricing subproblem, see, e.g., \citet{Irnich2010} and \citet{Desaulniers2020}, and (ii) column enumeration approaches, see, e.g., \citet{Agarwal1989}, \citet{Baldacci2008}, and \citet{Baldacci2023}.
In the first case, fixing decisions are applied directly to the PP on local structures of the pricing graph (e.g., arcs or arc sequences in path-based formulations). Removing such a structure implicitly eliminates all master variables whose columns contain it. However, these tests rely on sufficient conditions that ensure all columns with the structure have reduced cost above the target gap; if a structure is not eliminated, no guarantee is provided that all associated columns lie within the gap. In the second case, all columns corresponding to maximal stable sets with reduced cost not exceeding the gap are generated and used to build a reduced master problem, which is then solved as a MILP. A recent work by \citet{yang2024} further investigates advanced techniques to reduce the set of generated columns.

In this section, we describe a reduced-cost fixing approach along the lines of the second strategy. 
Given a target objective value \(\tau\), instead of explicitly enumerating all columns that pass the reduced-cost fixing tests, we construct a reduced ZDD that represents them implicitly. Unlike classical reduced-cost fixing in the pricing problem, where local fixing rules may leave some columns in the reduced pricing graph with reduced costs exceeding the target gap, our reduction is embedded directly into the recursive ZDD construction. The algorithm discards a branch either when the maximal stable set represented by that branch violates a fixing threshold, or earlier when all maximal stable sets extending the current partial branch can be certified to violate such a threshold. As formalized below, the final reduced ZDD contains only maximal stable sets that are not certified for elimination by the adopted fixing tests.
Additionally, we exploit alternative or multiple dual vectors to further strengthen the variable reduction and potentially improve the quality of the computed dual bounds.
The following lemma underlies the reduced-cost fixing strategy using a dual-feasible vector.

\begin{lemma}
\label{lemma:sc_fix}
Let \(\tau\in\mathbb{Z}_{\ge 0}\) be a target value, and let
\({\boldsymbol \pi}\in\mathbb{R}^{|\mathcal{V}|}_{\ge 0}\) be a dual feasible solution of the MP. For a maximal stable set \(\bar{\mathcal{S}}\in\mathscr{S}\), if
\begin{equation}
\label{eq:sc_fix_condition}
rc(\bar{\mathcal{S}},{\boldsymbol \pi})
>
\tau-\sum_{v\in\mathcal{V}}\pi_v,
\end{equation}
then no feasible integer solution of~\eqref{SC} with objective value at most
\(\tau\) can satisfy \(\xi_{\bar{\mathcal{S}}}=1\).
\end{lemma}
\begin{proof}{Proof}
    See \S\ref{sec:EC-ProofProp1} in the e-companion.\Halmos
\end{proof}
Lemma~\ref{lemma:sc_fix} eliminates columns that cannot improve the incumbent bound $UB$. Since the coloring objective is integer-valued, improving \(UB\) is equivalent to finding an integer solution of~\eqref{SC} with objective value at most \(\tau:=UB-1\). For a dual feasible solution \({\boldsymbol\pi}\) of the MP, the value \(\sum_{v\in\mathcal V}\pi_v\) provides a valid lower bound on the MP objective, and the term \(\tau-\sum_{v\in\mathcal V}\pi_v\) represents the remaining gap to the target value \(\tau\). Therefore, any variable \(\xi_{\bar{\mathcal S}}\) satisfying \eqref{eq:sc_fix_condition} can be fixed to zero for this target value. The ZDD construction below is defined with respect to \(\tau\).

\subsubsection{Generating Alternative or Multiple Dual Vectors.}\label{sec:genaltdual}

According to \citet{sellmann2004theoretical} and \citet{Mingozzi2013}, suboptimal or feasible dual solutions can strengthen the effect of reduced-cost fixing. Both \citet{De2023} and \citet{yang2024} design dual formulations to generate high-quality dual vectors for reduced-cost fixing. Their approaches are based on solving auxiliary optimization problems that maximize objective functions involving variable reduced costs and the dual bound, thereby seeking a favorable trade-off between increasing reduced costs and tightening the bound. However, these formulations aim at \emph{globally} strong dual vectors and do not explicitly encode the \emph{fixing condition} for any specific variable. The resulting dual vector may achieve a good overall trade-off while still failing to push a targeted variable beyond the fixing threshold, especially on instances with small optimality gaps, where proving optimality may hinge on eliminating variables whose reduced costs are close to the fixing threshold.

To address this issue, we introduce a new dual-based formulation that searches for diverse dual-feasible solutions tailored to certify the elimination of selected variables. For a given target value \(\tau\) and a maximal stable set \(\mathcal{S}\), or equivalently the associated variable \(\xi_{\mathcal{S}}\), Lemma~\ref{lemma:sc_fix} implies that it suffices to find a dual feasible solution \({\boldsymbol \pi}\in\mathbb{R}^{|\mathcal{V}|}_{\ge 0}\) such that
$rc(\mathcal{S},{\boldsymbol \pi})>\tau-\sum_{v\in\mathcal{V}}\pi_v$, which guarantees that \(\xi_{\mathcal{S}}=0\) in any feasible integer solution of~\eqref{SC} with objective value at most \(\tau\). In practice, given a subset \(\hat{\mathscr{S}} \subseteq \mathscr{S}\), we solve the following parametric optimization problem, which is also referred to as the dual formulation \(DF\) in the sequel, to obtain such a vector, using \(\varepsilon>0\) as a numerical tolerance and \(M\) as a weighting constant:
\begin{subequations}\label{dual_form}
\begin{align}
DF(\hat{\mathscr{S}}) \quad
&& \max_{{\boldsymbol \pi} \ge {\boldsymbol 0},{\boldsymbol \alpha} \ge {\boldsymbol 0}} \quad \sum_{v \in \mathcal{V}} \pi_v \;-\; M \sum_{\mathcal{S} \in \hat{\mathscr{S}}} \alpha_{\mathcal{S}}
\label{obj_df} \\
&& \sum_{v \in \mathcal{S}} \pi_v &\le 1,
&  \mathcal{S} \in \mathscr{S},
\label{cover_df} \\
&& 1-\sum_{v \in \mathcal{S}} \pi_v + \alpha_{\mathcal{S}}
&\ge
\tau - \sum_{v \in \mathcal{V}} \pi_v + \varepsilon,
&  \mathcal{S} \in \hat{\mathscr{S}}.
\label{diff_df}
\end{align}
\end{subequations}
In the objective~\eqref{obj_df}, \(\sum_{v\in\mathcal{V}}\pi_v\) is the dual lower bound of the RMP, while the slacks \(\alpha_{\mathcal{S}}\) measure how far each associated variable \(\xi_{\mathcal S}\) is from meeting the fixing threshold based on constraints \eqref{diff_df} which link each \(\alpha_{\mathcal{S}}\) to the reduced cost of \(\xi_{\mathcal{S}}\). Constraints \eqref{cover_df} impose dual feasibility. If \(\alpha_{\mathcal{S}}=0\), then \(\xi_{\mathcal{S}}\) can be safely excluded from any solution with objective value at most $\tau$; otherwise, the current dual solution \({\boldsymbol \pi}\) cannot certify its elimination. Hence, instead of rewarding reduced costs only through an aggregate objective, constraints \eqref{diff_df} directly enforce a threshold-driven mechanism, while the objective balances improvement of the dual bound against the total redundancy \(\sum_{\mathcal{S}\in\hat{\mathscr{S}}}\alpha_{\mathcal{S}}\). This design favors dual solutions that are not only globally strong but also specifically effective at pushing selected variables to the fixing threshold, thereby differing from previous dual models in explicitly incorporating the fixing condition.  The weight $M$ controls the trade-off between redundancy reduction and dual-bound improvement, while the tolerance $\varepsilon$ provides a safety margin and promotes diversity among dual solutions.

From an algorithmic standpoint, \(DF(\hat{\mathscr{S}})\) naturally supports the generation of a diverse pool of dual-feasible solutions that 
are useful not only for reduced-cost fixing but also for strengthening the lower bound used to certify whether a feasible integer solution with objective value at most \(\tau\) exists.
Following \citet{De2023}, we focus on basic variables in the optimal RMP basis, since fixing such variables forces a basis violation and empirically yields high-quality, complementary suboptimal dual solutions. 
Accordingly, we instantiate one \(DF(\hat{\mathscr{S}})\) model, defined in~\eqref{dual_form}, per targeted basic variable, seeking a dual solution that simultaneously pushes the targeted variable beyond the fixing threshold and maintains a strong dual objective value.
The benefit of this procedure is that some basic MP variables, corresponding to maximal stable sets that cannot be fixed by the optimal dual solution of the MP, may become certifiably irrelevant under the alternative dual solutions generated by \(DF(\hat{\mathscr{S}})\). After these variables are removed, reoptimizing the MP over the reduced pricing space may yield a stronger lower bound for certifying whether a feasible integer solution with objective value at most \(\tau\) exists.
In some cases, the improved lower bound may directly certify optimality by closing the remaining gap to the incumbent upper bound. Because \(DF(\hat{\mathscr{S}})\) contains a large family of covering constraints~\eqref{cover_df}, we solve it by row generation, using the same two-phase pricing oracle as in the root-node column generation. 
The heuristic pricer is used first to identify violated constraints. If none are found, the exact pricer is invoked, and row generation terminates only when it certifies that no negative reduced-cost maximal stable set exists. In this sense, the role of \(DF(\hat{\mathscr{S}})\) is twofold: it is both a fixing mechanism and a lower-bound strengthening device.

Model \(DF(\hat{\mathscr{S}})\) is used to generate alternative dual solutions that are feasible with respect to the dual of the MP. These feasible dual solutions can be directly used for reduced-cost fixing, as shown in Lemma~\ref{lemma:sc_fix}. In addition to these feasible dual solutions, we also exploit the dual vectors obtained during the root-node column generation process. 
Before column generation terminates, such dual vectors may still be infeasible for the full MP, since some columns with negative reduced cost may not yet have been generated. 
Nevertheless, these infeasible dual vectors can still provide valid fixing information under additional conditions. This motivates the following reduced-cost fixing lemma using a not necessarily dual-feasible vector.
\begin{lemma}
\label{lemma:infeasible_fix}
Let \(\tau\in\mathbb{Z}_{\ge 0}\) be a target value, and let
\({\boldsymbol \pi}\in\mathbb{R}^{|\mathcal{V}|}_{\ge 0}\) satisfy
\(rc^*({\boldsymbol \pi})<0\). If \(\bar{\mathcal S}\in\mathscr S\) satisfies
\begin{equation}
\label{eq:infeasible_fix_condition}
rc(\bar{\mathcal S},{\boldsymbol \pi})
>
\tau-\sum_{v\in\mathcal V}\pi_v-(\tau-1)rc^*({\boldsymbol \pi}),
\end{equation}
then no feasible integer solution of~\eqref{SC} with objective value at most \(\tau\)
satisfies \(\xi_{\bar{\mathcal S}}=1\).
\end{lemma}
\begin{proof}{Proof}
    See \S\ref{sec:EC-ProofProp2} in the e-companion.\Halmos
\end{proof}
Let $\bm\Pi_F$ and $\bm\Pi_I$ denote sets of feasible dual solutions and not necessarily dual-feasible vectors, respectively, and define the set of dual vectors for fixing as follows:
\[
\bm\Pi=\bm\Pi_F\cup\bm\Pi_I
=\{\boldsymbol\pi_1,\boldsymbol\pi_2,\ldots,\boldsymbol\pi_q\},
\qquad q:=|\bm\Pi|.
\]
For \(h\in\{1,2,\ldots,q\}\) and \(v\in\mathcal V\), let \(\pi_{h,v}\) denote the component of \(\boldsymbol\pi_h\) associated with vertex \(v\). For each \(\boldsymbol\pi_h\in\bm\Pi\), we associate a fixing threshold \(\delta_h\), defined as
\begin{equation}
\label{eq:delta_def}
\delta_h :=
\begin{cases}
\tau-\displaystyle\sum_{v\in\mathcal V}\pi_{h,v},
& \boldsymbol\pi_h\in\bm\Pi_F,\\[2mm]
\tau-\displaystyle\sum_{v\in\mathcal V}\pi_{h,v}-(\tau-1)rc^*(\boldsymbol\pi_h),
& \boldsymbol\pi_h\in\bm\Pi_I,
\end{cases}
\qquad h\in\{1,2,\ldots,q\}.
\end{equation}
We denote by \(\boldsymbol\delta=(\delta_1,\delta_2,\ldots,\delta_q)\) the vector collecting these fixing thresholds. 
Each vector \(\boldsymbol\pi_h\in\bm\Pi_F\) is dual feasible for the full MP. Each vector \(\boldsymbol\pi_h\in\bm\Pi_I\) satisfies \(\boldsymbol\pi_h\ge \boldsymbol 0\) and \(rc^*(\boldsymbol\pi_h)<0\), where \(rc^*(\boldsymbol\pi_h)\) denotes the minimum reduced cost over the full family \(\mathscr S\), or equivalently a certified global lower bound on the reduced costs of all columns in \(\mathscr S\).
For a maximal stable set \(\bar{\mathcal S}\in\mathscr S\), if there exists \(h\in\{1,2,\ldots,q\}\) such that
\(
rc(\bar{\mathcal S},\boldsymbol\pi_h)>\delta_h,
\)
then \(\bar{\mathcal S}\) cannot belong to any feasible integer solution of~\eqref{SC} with objective value at most \(\tau\).

\subsubsection{ZDD-Based Column Reduction.}\label{subsec:ZDD_reduce}
To reduce the ZDD while preserving optimality, we integrate dual-based reduced-cost fixing into the ZDD construction. Using the set of alternative dual vectors and fixing thresholds \((\bm\Pi,\boldsymbol\delta)\), we discard maximal stable sets certified as non-improving with respect to the target objective value \(\tau\). More importantly, whenever such a certificate applies to all feasible completions of a partial branch, the corresponding ZDD subtree is pruned altogether.
We formalize this early-elimination test as a pruning rule. Theorem~\ref{thm:pruning_rule} establishes a monotonicity property that justifies the pruning and enables a \emph{reduced} ZDD containing only promising maximal stable sets.
\begin{thm}
\label{thm:pruning_rule}
Let
\(\bm\Pi=\{\boldsymbol\pi_1,\boldsymbol\pi_2,\ldots,\boldsymbol\pi_q\}\) and let
\(\boldsymbol\delta=(\delta_1,\delta_2,\ldots,\delta_q)\) be defined by
\eqref{eq:delta_def}. Let \(\bar{\mathcal S}\) be a partial stable set and let
\(\mathcal C\subseteq\mathcal V\) be such that every maximal stable set
\(\hat{\mathcal S}\) generated below the current branch satisfies
\(
\bar{\mathcal S}\subseteq\hat{\mathcal S}\subseteq\bar{\mathcal S}\cup\mathcal C.
\)
If there exists \(h\in\{1,2,\ldots,q\}\) such that
\begin{equation}
	\label{eq:pruning_condition}
	1-\sum_{v\in\bar{\mathcal S}\cup\mathcal C}\pi_{h,v}>\delta_h,
\end{equation}
then no feasible integer solution of~\eqref{SC} with objective value at most \(\tau\) can satisfy \(\xi_{\hat{\mathcal S}}=1\).
\end{thm}
\begin{proof}{Proof}
    See \S\ref{sec:EC-ProofProp3} in the e-companion.\Halmos
\end{proof}
Figure~\ref{fig:comparison_zdd}(c) shows the reduced ZDD constructed with target value \(\tau=2\), dual vector \(\boldsymbol\pi=(1,0,0,1,0,0)\), and fixing threshold \(\delta=0\). 
The goal is to build a ZDD that preserves all maximal stable sets that may participate in a feasible solution of~\eqref{SC} using at most two colors, while pruning those certified as irrelevant by reduced-cost fixing. The complete ZDD in Figure \ref{fig:comparison_zdd}(b) encodes four maximal stable sets. The reduced ZDD prunes node \(e\) and all branches that cannot participate in any feasible solution of value at most \(\tau\), retaining only the two paths corresponding to \(S_1\) and \(S_2\). The final reduced ZDD has size 20, consisting of six decision nodes, two terminal nodes, and twelve arcs; see \S\ref{sec:app_zdd_illustration} for the construction details.

\section{Overview of \BPname}\label{sec:algoverview}

This section presents the main components of \BPname\ and explains how they are integrated within the proposed BP framework. \BPname\ proceeds in five main stages. It first applies a preprocessing phase to reduce the size of graph $\mathcal{G}$ and, whenever possible, certify optimality directly (\S\ref{subsec:preprocessing}). Next, high-quality primal solutions are computed using heuristic algorithms that provide an initial upper bound and an initial pool of columns, each associated with a feasible maximal stable set,  for the CG process (\S\ref{subsec:primal}). The algorithm then solves the root node of the BP tree by CG, based on the computation of safe lower bounds and early termination criteria (\S\ref{subsec:safe_bounds}) and using a two-phase pricing strategy for the weighted maximum stable set problem in order to compute strong lower bounds (\S\ref{subsec:lb_zdd}). After the root node is solved, \BPname\ attempts to construct either a complete or a reduced ZDD, which is subsequently used as an exact pricing structure in the branching nodes. The search proceeds through a branching strategy based on fractional variables, with updates of the RMP and the ZDD-based pricer (\S\ref{subsec:branching}). Section \ref{sec:ZDDconstr} describes how ZDDs are constructed. 

\subsection{Preprocessing Strategies}\label{subsec:preprocessing}

\BPname\ begins with a preprocessing phase that reduces the number of vertices while preserving correctness, and may even certify optimality before the BP procedure starts. We employ two classical vertex-reduction techniques, as described in \citet{lucet2004pre}, and apply them iteratively until no further reduction is possible. The first technique removes dominated vertices: given two distinct vertices \(u\) and \(v\), if every vertex adjacent to \(u\) is also adjacent to \(v\), then \(u\) is dominated and can be deleted. The second technique removes low-degree vertices: any vertex with degree strictly less than a certified lower bound on the chromatic number can be deleted. After each removal, degrees and neighborhoods are updated in the reduced graph. Moreover, if at any point the reduced graph has at most as many vertices as the certified lower bound, then optimality is established immediately.
To make these reductions effective, we compute a strong certified lower bound on the chromatic number $\chi(\mathcal{G})$ as the maximum of two complementary bounds. The first is the clique bound, given by the size of a maximum clique, computed with \texttt{CliSAT}~\citep{Pablo2022}, a state-of-the-art exact algorithm for the Maximum Clique Problem, with a time limit of 1 second. The second is the Mycielski-based bound introduced by \citet{hebrard2020constraint}, which exploits pseudo-Mycielskian structures and can certify larger chromatic numbers even when the maximum clique is small. This technique is also used for reductions and optimality certification in \citet{brand2026customized}. We compute this bound with the greedy heuristic of \citet{hebrard2020constraint}, which tries to identify the Mycielskian of a subgraph and can be applied iteratively to obtain stronger bounds. 
After applying the vertex-reduction rules, we also compute a vertex ordering used in the subsequent ZDD construction. We follow the maximal path decomposition ordering of \citet{morrison2014characteristics}, also adopted by \citet{Morrison2016}. The method partitions the remaining vertices into vertex-disjoint paths, where each path is a sequence of distinct vertices such that consecutive vertices are adjacent, and orders the vertices by their appearance along these paths. In our implementation, each path starts from a minimum-degree vertex in the current residual graph and is extended by repeatedly selecting an unvisited neighbor of the last vertex with minimum degree. Once no further extension is possible, the path is closed, its vertices are removed, and the procedure continues on the residual graph until all vertices have been ordered.

\subsection{Computing Primal Solutions}\label{subsec:primal}

After the preprocessing phase, \BPname\ executes a heuristic algorithm from the literature to compute high-quality initial colorings and to establish an initial upper bound on the chromatic number. These solutions also populate the initial column pool with a feasible coloring.
Specifically, \BPname\ employs the first phase of the MMT heuristic proposed by \citet{malaguti2008metaheuristic}, an evolutionary algorithm based on tabu search. The procedure starts with an initial upper bound computed by the Dsatur heuristic~\citep{brelaz1979new} and then searches for feasible colorings with progressively fewer colors, moving toward a lower bound given by the size of a maximal clique. The heuristic terminates either when the prescribed time limit is reached or when no feasible coloring can be found for a target number of colors.

The time allocated to each MMT iteration is scaled based on the graph size after preprocessing, and a global time limit is imposed on the entire heuristic phase. Each tabu-search run follows the parameter setting proposed by \citet{malaguti2008metaheuristic}, including the maximum number of iterations, the initial population size, and the tabu tenure. The implementation of the first phase of the MMT heuristic was kindly provided by the authors of \citet{malaguti2008metaheuristic} via personal communication.

\subsection{Numerically Safe Lower Bounds and Early Termination Criteria} \label{subsec:safe_bounds}

Most modern LP solvers, including CPLEX used by \BPname, rely on floating-point arithmetic and may return dual solutions affected by small numerical errors.
To obtain numerically safe lower bounds at any branching node, we adopt the scaling-and-rounding technique of~\citet{Held2012}, which has also been used by~\citet{Baldacci2023} recently for the Bin Packing Problem. 

Let ${\boldsymbol \pi}\in\mathbb{R}^{\mathcal{V}}_{\ge 0}$ be an optimal dual solution of the RMP at the current CG iteration, and let $k>0$ be an integer scaling factor.
We define an integer vector $\check{\boldsymbol \pi}\in\mathbb{Z}^{\mathcal{V}}_{\ge 0}$ by setting $\check{\pi}_v=\lfloor k\,\pi_v\rfloor$ for all $v\in \mathcal{V}$.
By construction, for every $v\in {\mathcal{V}}$,
$\pi_v-\frac{1}{k}\le\frac{\check{\pi}_v}{k}\le\pi_v,$
and thus
$
\sum_{v\in\mathcal{V}_u}\pi_v-\frac{|\mathcal{V}_u|}{k}
\ \le\
\sum_{v\in\mathcal{V}_u}\frac{\check{\pi}_v}{k}
\ \le\
\sum_{v\in\mathcal{V}_u}\pi_v,
$
so that $\check{\boldsymbol \pi}/k$ is a componentwise underestimator of ${\boldsymbol \pi}$ and reduced-cost tests performed with $\check{\boldsymbol \pi}$ are conservative and can be carried out in exact integer arithmetic.
Without loss of optimality, we set $\check{\pi}_v=0$ for all $v\in\mathcal{V}_c$.
We then solve the PP using $\check{\boldsymbol \pi}$ so that all intermediate computations are performed in integer arithmetic.
Variables with negative \emph{scaled} reduced cost, i.e., $k-\sum_{v\in\mathcal{S}}\check{\pi}_v<0$, are generated and added to the RMP until no such variable exists.

At termination of this process, $\check{\boldsymbol \pi}/k$ is dual feasible for the LP relaxation of~\eqref{SC} at the current branching node; hence,
$LB_T(\check{\boldsymbol \pi})\;:=\;|\overline{\mathscr{S}}_1|\;+\;\left\lceil \sum_{v\in\mathcal{V}_u}\frac{\check{\pi}_v}{k}\right\rceil$
is a numerically safe integer lower bound on the number of stable sets (colors) required, already accounting for the variables whose lower bound is set to~1, and can be compared directly with the incumbent value $UB$ for pruning.

During CG, even before reaching dual feasibility with respect to all variables, we compute an additional numerically safe lower bound based on the classical result of~\citet{Farley1990} for set-covering problems; see also~\citet{Baldacci2023}.
Let $\check{\boldsymbol \pi}\in\mathbb{Z}^{\mathcal{V}}_{\ge 0}$ be the scaled integer dual vector used in pricing at the current CG iteration, and let $\widetilde{\mathcal{S}}$ be an optimal solution of the PP obtained with weights $\check{\boldsymbol \pi}$.
If $\sum_{v\in\widetilde{\mathcal{S}}}\check{\pi}_v>k$, we define the \emph{Farley bound} at the current branching node as
\begin{equation}
\label{eq:farley_bound}
LB_F(\check{\boldsymbol \pi})
\;:=\;
|\overline{\mathscr{S}}_1|
\;+\;
\left\lceil
\frac{\sum_{v\in\mathcal{V}_u}\check{\pi}_v}{\sum_{v\in\widetilde{\mathcal{S}}}\check{\pi}_v}
\right\rceil.
\end{equation}
This bound follows from a scaling argument.
The PP returns a maximal stable set $\widetilde{\mathcal{S}}$ maximizing $\sum_{v\in\mathcal{S}}\check{\pi}_v$ over all maximal stable sets allowed by branching; hence, for every such $\mathcal{S}$,
$
\sum_{v\in\mathcal{S}}\check{\pi}_v
\;\le\;
\sum_{v\in\widetilde{\mathcal{S}}}\check{\pi}_v.
$
Therefore, the rescaled vector $\check{\boldsymbol\pi}\big/\sum_{v\in\widetilde{\mathcal{S}}}\check{\pi}_v$ is dual feasible for the dual of the LP relaxation of~\eqref{SC} at the current branching node (i.e., feasible with respect to all dual constraints induced by maximal stable sets not forbidden by branching), and its dual objective value equals
$\sum_{v\in\mathcal{V}_u}\check{\pi}_v \big/ \sum_{v\in\widetilde{\mathcal{S}}}\check{\pi}_v$.
Adding $|\overline{\mathscr{S}}_1|$ and rounding up yields the lower bound~\eqref{eq:farley_bound}, which can be compared directly with the incumbent value $UB$ for pruning.

Since $LB_F(\check{\boldsymbol \pi})$ is valid even when improving variables still exist, it can be used to accelerate CG and to prune branching nodes safely.
Early CG termination is triggered when
$LB_F(\check{\boldsymbol \pi})
\;=\;
|\overline{\mathscr{S}}_1|
\;+\;
\left\lceil \sum_{v \in \mathcal{V}_u} \frac{\check{\pi}_v}{k} \right\rceil$,
in which case further CG iterations cannot improve the resulting integer lower bound at the current branching node. Moreover, safe branching node fathoming is applied when $LB_F(\check{\boldsymbol \pi}) \ge UB$, as the Farley bound is an integer lower bound already accounting for variables whose lower bound is set to~1, and therefore can be compared directly with the incumbent upper bound.
These criteria complement the termination based on the safe dual-feasible bound of~\citet{Held2012} and are applied throughout the branching tree of \BPname. 

\subsection{Pricing Strategies}\label{subsec:lb_zdd}

\BPname\ employs two main pricing strategies: at the root node of the BP tree (\S\ref{sec:PP-root}), where column generation is used to compute a strong lower bound, and at the branching nodes (\S\ref{sec:PP-enum}), where a ZDD-based exact pricer is applied once branching decisions alter the structure of the PP.

\subsubsection{Root-Node Pricing Strategy.} \label{sec:PP-root}

At the root node, the PP aims at identifying maximal stable sets with negative reduced cost and corresponds to the MWSSP, where vertex weights are obtained from the dual solution of the RMP at each CG iteration (see \S\ref{sec:SC}). Specifically, the dual values are multiplied by the scaling factor and rounded down to obtain integer weights, which are then used to compute numerically safe reduced costs and lower bounds (see \S\ref{subsec:safe_bounds}).

Since efficient heuristic and exact algorithms are available for the MWSSP, \BPname\ adopts a two-phase pricing scheme. At each pricing iteration, a heuristic algorithm is invoked to find a variable with a negative reduced cost. Only if the heuristic fails to identify such a variable is an exact algorithm invoked to solve the PP to optimality. If the exact pricer generates a negative reduced-cost variable, the next pricing iteration restarts with the heuristic pricer; otherwise, the column-generation process terminates.

As a heuristic pricer, \BPname\ employs \texttt{MN/TS}~\citep{wu2012multi}, a multi-start tabu search algorithm for the MWSSP.
The algorithm searches for a stable set with weight at least \(k+1\), corresponding to a variable with negative reduced cost. Once such a stable set is found during a tabu-search run, the current run is allowed to finish, but no further restarts are performed, and the heuristic pricer terminates.
As the exact pricer, \BPname\ employs \texttt{TSM-MWC}~\citep{jiang2018two}, a branch-and-bound algorithm for the MWSSP, used with its default parameter settings. To reduce the computational burden of both pricing procedures, vertices with zero dual value are discarded before solving the MWSSP. Whenever a negative reduced-cost stable set is identified, a vertex-insertion procedure is applied to extend it into a maximal stable set, and the corresponding column is then added to the RMP.

During the root-node column generation process, we collect intermediate RMP dual vectors and add them to \(\bm\Pi_I\). Specifically, an intermediate dual vector \(\boldsymbol{\pi}\) is retained only if the exact pricer has been called at the corresponding CG iteration, so that the value \(rc^*(\boldsymbol{\pi})\) is known exactly. This condition is necessary for safely applying Lemma~\ref{lemma:infeasible_fix}. Dual vectors for which only the heuristic pricer has been applied are not used for reduced-cost fixing. Since the CG process may generate a large number of candidate intermediate dual vectors, we keep at most 10 vectors in \(\bm\Pi_I\), i.e., \(|\bm\Pi_I|\le 10\). We retain the 10 most recently generated vectors, because those produced in later CG iterations typically yield tighter lower bounds and are thus more informative for subsequent fixing. When the root-node CG terminates with exact pricing certification, the final dual vector is dual-feasible for the MP and is added to \(\bm\Pi_F\). After the root-node CG process terminates, if the fixing threshold associated with the converged dual feasible solution of the MP is smaller than 1, the set of dual feasible solutions is further enriched with alternative dual feasible solutions returned by model \(DF\); these dual feasible solutions are also added to \(\bm\Pi_F\). The resulting set \(\bm\Pi=\bm\Pi_F\cup\bm\Pi_I\) is then used in the fixing procedure, with vectors in \(\bm\Pi_F\) and \(\bm\Pi_I\) handled according to Lemma~\ref{lemma:sc_fix} and Lemma~\ref{lemma:infeasible_fix}, respectively.

\subsubsection{Enumeration-Tree Pricing Strategy.}\label{sec:PP-enum}

After column generation terminates at the root node, \BPname\ attempts to construct a ZDD, which is subsequently used as an exact pricing data structure in the branching nodes. Once branching decisions are imposed on the variables of the SC formulation, the root-node pricing strategy cannot be applied directly, as the pricing problem structure changes. Depending on the quality of the upper bound \(UB\) and lower bound \(LB\) computed at the root node, \BPname\ constructs either a \emph{complete ZDD} or a \emph{reduced ZDD}. 
The reduced ZDD is built with respect to the target value \(\tau=UB-1\). Since the objective value of~\eqref{SC} represents the number of colors and is therefore integer-valued, finding a solution strictly better than the incumbent \(UB\) is equivalent to finding a feasible integer solution with objective value at most \(UB-1\) (if any). Thus, the goal at this stage is to determine whether there exists a feasible integer solution of cost at most \(\tau\). If such an improving solution exists, the reduced ZDD preserves the maximal stable sets that may participate in a solution of cost at most \(\tau\), thereby reducing the search space and accelerating the search for a better incumbent. If no such improving solution exists, the reduced ZDD provides a restricted but valid search space for certifying that no solution of value at most \(\tau\) exists, and hence that the incumbent value \(UB\) is optimal. Accordingly, maximal stable sets that are certified as unable to appear in any feasible integer solution of cost at most \(\tau\) are removed during the reduction.

In both cases, pricing is performed by assigning each high arc its corresponding scaled dual weight and each low arc a weight of zero, so that the pricing problem reduces to finding a longest root-to-\textsc{True} path in the ZDD, yielding a maximal stable set of minimum reduced cost. In the reduced-ZDD case, the root-node column-generation process is restarted over the reduced pricing space in order to recompute the corresponding lower bound. Since eliminated variables are certified irrelevant for any feasible integer solution with objective value at most \(\tau\), reoptimizing the root-node column generation over the reduced pricing space remains valid for certifying whether such a solution exists and may strengthen the corresponding lower bound.

\subsection{Branching Strategy}\label{subsec:branching}

At the root node, \BPname\ solves the pricing problem by a two-phase strategy that combines heuristic and exact algorithms for the MWSSP. After branching, however, the additional constraints alter the structure of the pricing problem, so these algorithms cannot be applied directly. For this reason, \BPname\ adopts a hybrid pricing strategy: the two-phase method is used at the root node, while the ZDD-based pricer is employed at the subsequent nodes of the BP tree.

Consider a generic node of the enumeration tree, where the current MP has been solved and yields an optimal solution \(\bm{\xi}^*\). If \(\bm{\xi}^*\) is integral, the node is fathomed. Otherwise, we branch on a variable \(\xi_{\mathcal{S}}\) associated with a maximal stable set \(\mathcal{S}\in\overline{\mathscr{S}}\) such that \(0<\xi_{\mathcal{S}}^*<1\). In our implementation, \(\xi_{\mathcal{S}}\) is selected as the most fractional variable in the current solution. Branching is then performed on \(\xi_{\mathcal{S}}\), generating two child nodes: in the left child, the variable \(\xi_{\mathcal{S}}\) is enforced by setting \(\xi_{\mathcal{S}}=1\), whereas in the right child it is forbidden by setting \(\xi_{\mathcal{S}}=0\). 
Accordingly, the maximal stable set \(\mathcal{S}\) associated with the branching variable is added to \(\overline{\mathscr{S}}_1\) in the left child and to \(\overline{\mathscr{S}}_0\) in the right child.
These branching decisions are enforced consistently in both the restricted master problem and the ZDD-based pricing procedure, as described in Section~\ref{subsec:ColumnGenerationBranchingScheme}.

Excluding variables may render the RMP temporarily infeasible if some vertex constraints are no longer covered by any remaining maximal stable set. When necessary, \BPname\ restores feasibility by adding artificial singleton columns for the affected vertices, with the same covering structure as singleton stable sets. These columns are introduced only to re-establish RMP feasibility during reoptimization and are handled consistently within the pricing/bounding scheme.

At the root node, the RMP is solved using the primal simplex method, which is well-suited to column-generation reoptimization after adding new columns. After branching, child nodes differ from their parent mainly through bound changes, and the RMP is reoptimized by dual simplex from the inherited basis. The branching tree is explored by a depth-first search strategy.

\subsection{ZDD Construction}\label{sec:ZDDconstr}

\BPname\ employs the recursive algorithm proposed in \citet{Morrison2016} to construct the complete ZDD denoted by $\mathcal{Z}_\mathscr{S}$ representing all maximal stable sets in $\mathcal{G}$. Specifically, the construction algorithm implicitly enumerates all possible maximal stable sets in the graph $\mathcal{G}=(\mathcal{V},\mathcal{E})$ while simultaneously constructing the ZDD graph $\mathcal{Z}_\mathscr{S}=(\mathcal{N},\mathcal{A})$ through recursive calls to a procedure named $MakeIndSetZDD$. During each iteration of $MakeIndSetZDD$, it expands a partially constructed stable set $\mathcal{R}$ using a set $\mathcal{U}=\{u_i,...,u_m\}$ of uncovered vertices. Vertices in $\mathcal{U}$ can be inserted into the partial stable set $\mathcal{R}$ to produce a larger stable set. To further extend the path induced by $\mathcal{R}$ in the ZDD graph using vertices from $\mathcal{U}$, the algorithm recursively constructs a ZDD subgraph with these vertices, inserting the ZDD nodes from this subgraph right after the ZDD path associated with $\mathcal{R}$. The algorithm constructs ZDD nodes corresponding to $\mathcal{U}$ through two recursive calls: in the high branch, it expands the stable set by adding $u_i$ to $\mathcal{R}$ and further explores a smaller set obtained by removing $u_i$ and its neighbors $N[u_i]$ from $\mathcal{U}$. In the low branch,  $u_i$ is forbidden from being used in $\mathcal{R}$, and the algorithm explores a smaller set obtained by omitting $u_i$ from $\mathcal{U}$.
To reduce the size of $\mathcal{Z}_\mathscr{S}$, the algorithm introduces a hash table to store nodes already inserted into $\mathcal{Z}_\mathscr{S}$. When a newly generated node has the same associated vertex, low child, and high child as a node in the hash table, it directly returns the existing node from the hash table to merge the nodes (see \citet{Morrison2016} for related details).

\begin{algorithm}[t]
\caption{MakeReducedIndSetZDD($U, i, RC$)}\label{makeIS}
\linespread{1.2}
\fontsize{9pt}{9pt}
\selectfont
\SetKw{KwGoTo}{go to}
\SetKwInOut{Parameter}{Parameters}
\LinesNumbered
\SetAlgoLined
\SetAlgoNoEnd

\KwIn{A set $U=\{u_1,u_2,\dots,u_m\}$ of uncovered vertices such that $u_j<u_{j+1}$ with respect to the vertex ordering on \(\mathcal V\), a current index \(i\), and an array $RC$ storing the reduced costs of the current stable set with respect to the dual-vector set $\bm\Pi$.}

\KwOut{The root node of a ZDD characterizing candidate maximal stable sets
in $\mathcal{G}[U]$ that may yield improving solutions and can be formed with
vertices in $\{u_i,u_{i+1},\dots,u_m\}$.}

\uIf{set $U_1=\{u_1,u_2,\dots,u_{i-1}\}$ has a vertex with no neighbor in $U_2=\{u_i,u_{i+1},\dots,u_m\}$}{
    \KwRet{$0$}\tcp*[r]{Maximality cannot be achieved}
}

\If{$U=\emptyset$}{
    \For{$h \gets 1$ \KwTo $q$}{
        \uIf{$RC[h]>\delta_h$}{
            \KwRet{$0$}\tcp*[r]{Fixed by reduced-cost fixing}
        }
    }
    \KwRet{$1$}\tcp*[r]{A maximal stable set is accepted}
}

$U_H \gets U \setminus N[u_i]$\tcp*[r]{High branch includes $u_i$}
$RC_H[h] \gets RC[h]-\pi_{h,u_i}$ for all $h \in\{1,2,\ldots,q\}$\tcp*[r]{Update reduced costs}
$l \gets \min\{j>i:u_j\in U_H\}$, or $|U_H|+1$ if none exists \tcp*[r]{\(j\) is indexed in \(U_H\)}

\eIf{\texttt{PruneByProp}$(RC_H,U_H,l,\bm\Pi,\boldsymbol\delta)$}{
    $z_h \gets 0$\tcp*[r]{Prune high subtree}
}{
    $z_h \gets \texttt{MakeReducedIndSetZDD}(U_H,l,RC_H)$\tcp*[r]{Continue high branch}
}

\eIf{\texttt{PruneByProp}$(RC,U,i+1,\bm\Pi,\boldsymbol\delta)$}{
    $z_l \gets 0$\tcp*[r]{Prune low subtree}
}{
    $z_l \gets \texttt{MakeReducedIndSetZDD}(U,i+1,RC)$\tcp*[r]{Continue low branch}
}

\uIf{$z_h=0$}{
    \KwRet{$z_l$}\tcp*[r]{ZDD reduction rule}
}

\uIf{$\exists z\in \mathcal{Z}_{\mathscr{S}}$ such that $v(z)=i$, $lo(z)=z_l$, and $hi(z)=z_h$}{
    \KwRet{$z$}\tcp*[r]{Reuse equivalent node}
}

\KwRet{$\mathcal{Z}_{\mathscr{S}}.\mathrm{insert}(i,z_l,z_h)$}\tcp*[r]{Create new ZDD node}

\end{algorithm}

To implement the ZDD reduction, we adapt the recursive procedure of \citet{Morrison2016} by integrating reduced-cost fixing to build a reduced ZDD \(\mathcal{Z}_{\mathscr{S}'}\) of potentially maximal stable sets in \(\mathcal{G}[U]\), where \(\mathcal{G}[U]\) denotes the subgraph of \(\mathcal{G}\) induced by the vertex set \(U\). 

Algorithm~\ref{makeIS}, referred to as \(\texttt{MakeReducedIndSetZDD}\), takes as input the uncovered set \(U=\{u_1,\dots,u_m\}\), the current index \(i\), and the reduced-cost array \(RC\). During the recursion, for each \(h\in\{1,2,\ldots,q\}\), \(RC[h]\) stores the reduced cost of the current partial stable set with respect to the dual vector \(\boldsymbol\pi_h\). At the root call, since the partial stable set is empty, each entry of \(RC\) is initialized to \(1\). Hence, when vertex \(u_i\) is selected, the reduced-cost array is updated by setting \(RC_{H}[h]=RC[h]-\pi_{h,u_i}\). Let \texttt{PruneByProp}$(RC,U,i,\bm\Pi,\boldsymbol\delta)$ denote the pruning test induced by Theorem~\ref{thm:pruning_rule}, where the candidate set \(\mathcal C\) in the theorem is instantiated as the subset of \(U\) consisting of vertices from the current index onward, i.e., \(\mathcal C=\{u_j\in U: j\ge i\}\). A feasibility test prunes early (lines 1--2): if $U_1$ contains a vertex with no neighbors in $U_2$, maximality cannot be achieved, and the false terminal is returned. When $U=\emptyset$, a reduced-cost check is applied (lines 3--7): if any entry of \(RC\) exceeds the corresponding fixing threshold in $\boldsymbol\delta$, the branch is cut, implementing reduced-cost fixing. The recursion branches on $u_i$. In the high branch, $u_i$ and its neighbors are removed, \(RC\) is updated w.r.t.\ $\bm\Pi$, and recursion continues (lines 8--14). In the low branch, $u_i$ is excluded (lines~15--18). Theorem~\ref{thm:pruning_rule} is applied in both cases for early pruning.
To keep the ZDD compact, nodes with an empty high branch are collapsed to their low child (lines 19--20). Canonical representation is ensured via hashing: existing equivalent nodes are reused, otherwise new ones are created (lines 21--23).

We define the family of retained maximal stable sets as
\begin{equation}
\label{eq:sred_def}
\mathscr S_{\rm red}:=
\left\{
\mathcal S\in\mathscr S:
rc(\mathcal S,\boldsymbol\pi_h)\le\delta_h
\text{ for every } h\in\{1,2,\ldots,q\}
\right\}.
\end{equation}
The following corollary justifies using $\mathcal{Z}_{\mathscr{S}_{\rm red}}$ 
as the pricing structure in the subsequent BP phase: all maximal stable sets eliminated by reduced-cost fixing are guaranteed to be excluded from the ZDD.

\begin{corollary}\label{cor:reduced_zdd_correctness}
	Assume Algorithm~\ref{makeIS} is initialized with
	\(RC[h]=rc(\emptyset,\boldsymbol\pi_h)=1\) for every
	\(h\in\{1,2,\ldots,q\}\), updates \(RC[h]\) according to the selected high-branch
	vertices, and applies the pruning rule of Theorem~\ref{thm:pruning_rule}.
	Then the returned ZDD encodes exactly the maximal stable sets in
	\(\mathscr S_{\rm red}\) defined in~\eqref{eq:sred_def}.
\end{corollary}
\begin{proof}{Proof}
    See \S\ref{sec:EC-ProofCor1} in the e-companion.\Halmos
\end{proof}

The computations required to solve formulation $DF(\hat{\mathscr{S}})$ \eqref{dual_form} and the implementation of Algorithm~\ref{makeIS}, based on Lemma~\ref{lemma:infeasible_fix}, are also performed in a numerically safe manner using scaling techniques (see \S\ref{subsec:safe_bounds}). In particular, the set of alternative dual solutions $\bm\Pi$ and the associated threshold vector $\boldsymbol\delta$ are represented in scaled integer form. This ensures that all reduced-cost comparisons are numerically safe and can be carried out using exact integer arithmetic. Implementation details are omitted for brevity.

\section{Computational Study}\label{sec:compstud}

This section presents a computational study evaluating the performance of \BPname\ on the DIMACS and Erd\H{o}s--R\'enyi benchmark sets for the GCP. It begins by describing the benchmark instances, computational environment, main parameter settings, and the state-of-the-art approaches used for comparison (\S\ref{subsec:benchmark} and \S\ref{subsec:environment_param}). Sections \ref{subsubsec:comparison_no_code} and \ref{subsubsec:comparison_with_code} then report comparisons with literature results, distinguishing between approaches with and without publicly available code. Finally, Section \ref{subsec:effect_components} analyzes the impact and effectiveness of the main phases of \BPname.

\subsection{Benchmark Instances}\label{subsec:benchmark}

The DIMACS instances, introduced during the Second DIMACS Implementation Challenge on \textit{Cliques, Coloring, and Satisfiability}~\citep{johnson1996cliques}, comprise 137 GCP instances and have served for over 30 years as a standard benchmark for evaluating both heuristic and exact algorithms. 
Despite significant algorithmic progress in recent years, many large and difficult DIMACS instances remain unsolved and continue to serve as open benchmarks that drive further research on the GCP.

To analyze and present the features of these 137 DIMACS instances, see Table~\ref{tab:families} in the e-companion for details, we grouped them into 22 families based on similar structural characteristics, using their name prefixes as the main grouping criterion.  
The testbed is highly diverse in size: the smallest instance has 11 vertices, while the largest has 10{,}000; the number of edges ranges from 20 to over four million. 

In addition to the classical DIMACS benchmark set, we use the Erd\H{o}s--R\'enyi random graph instances of~\citet{brand2026customized} to evaluate performance on graphs with varying densities. The instances follow the classical \(G(n,p)\) model~\citep{erdds1959random}, with \(n \in \{70,80,90,100,110\}\), \(p \in \{0.05,0.1,0.15,0.2,0.25,0.3,0.4,0.5,0.7,0.9\}\), and 100 graphs for each parameter pair, for a total of 5,000 instances. 
We report results for \BPname\ on the full set. For the controlled comparison with publicly available algorithms, we use the first 10 instances from each parameter group, yielding a representative subset of 500 instances that still covers sparse to dense graphs. The original instance set is available at \url{https://zenodo.org/records/17328845}.

\subsection{Computational Environment and State-of-the-art Approaches}\label{subsec:environment_param}

This section first describes the computational setting of \BPname\ (\S\ref{sec:compenv}) and then briefly presents the state-of-the-art approaches used to compare the performance of \BPname\ (\S\ref{subsec:comparison}).

\subsubsection{Computational Environment and Parameter Settings of \BPname.}\label{sec:compenv}

\BPname\ was implemented in Java, and all experiments were conducted on a Windows machine equipped with an Intel i7-12700 processor (2.1 GHz), 32 GB of RAM, and Java version 17.0.13. Our computational platform requires approximately 3.00 CPU seconds to execute the DIMACS MC Machine Benchmark on instance \texttt{r500.5} (see \url{http://archive.dimacs.rutgers.edu/pub/dsj/clique/}). The RMP arising in the column-generation process (see \S\ref{sec:SC}) and model $DF(\hat{\mathscr{S}})$ \eqref{dual_form} are handled by the LP solver of IBM CPLEX 12.8 in single-thread mode.

We provide a public repository containing all benchmark instances, their corresponding computational results, and the main implementation details, including an executable version of \BPname\ for testing and validation, along with the optimal colorings computed for instances solved to proven optimality.

The main parameter settings of \BPname\ are summarized as follows. For numerically safe reduced-cost computations and lower-bound evaluation, discussed in Section~\ref{subsec:safe_bounds}, the dual values are scaled by the factor \(k=10^{5}\). In the dual-based formulation \(DF(\hat{\mathscr{S}})\) of Section~\ref{sec:genaltdual}, the penalty weight and tolerance are set to \(M=10^{4}\) and \(\varepsilon=0.05\), respectively. In addition, at most 10 dual vectors generated during the root-node CG process are retained for the subsequent column-enumeration phase described in Section~\ref{subsec:ZDD_reduce}.
In the preprocessing phase, described in Section~\ref{subsec:preprocessing}, \texttt{CliSAT} is run with a time limit of 1 second. For the computation of initial heuristic colorings, \BPname\ applies the first phase of the MMT heuristic to all instances after preprocessing. The procedure starts from an initial upper bound obtained by \textsc{Dsatur} and then attempts to find feasible colorings with progressively fewer colors. For each target number of colors, the tabu-search procedure is limited to 10,000 iterations, with an initial population size of 10 and a tabu tenure of 45. The time limit for each target-color search is set to \(0.1|\mathcal{V}|\) seconds when the reduced graph has at most 400 vertices, and to 100 seconds otherwise. In addition, a global time limit of 100 seconds is imposed on the whole heuristic phase. The heuristic terminates either when the global time limit is reached or when no feasible coloring can be found for a target number of colors.

At the root node, within the pricing scheme of Section~\ref{subsec:lb_zdd}, the heuristic MWSSP pricer \texttt{MN/TS} is used to search for a stable set with a weight at least \(k+1\), i.e., a variable with negative reduced cost. Once such a stable set is found, the current tabu-search run is allowed to finish, but no further restarts are performed, and the heuristic pricer terminates. Its tabu-search depth is set to \(10(1-\mu)|\mathcal{V}|\), where \(\mu\) denotes the edge density of the graph, and the maximum number of restarts is fixed to \(|\mathcal{V}|/2\). The exact MWSSP pricer \texttt{TSM-MWC} is used with its default parameter settings. 

\subsubsection{State-of-the-art Approaches.}\label{subsec:comparison}

Two groups of methods are considered separately, depending on whether their source code is publicly available. The first group consists of exact algorithms whose source code is not publicly available. These methods are relevant because they represent important state-of-the-art approaches based on different exact-solution paradigms. In particular, the method of \citet{van2022graph} relies on a network-flow ILP formulation combined with decision diagrams, whereas the methods of \citet{malaguti2011exact}, \citet{Gualandi2012}, and \citet{Morrison2016} belong to the BP family. Among them, \citet{Morrison2016} is especially close in spirit to the present work, since it also integrates decision diagrams into the pricing phase. Since these algorithms were tested on different subsets of DIMACS instances and under different time limits, a fully controlled comparison is not possible. Therefore, we report an aggregate comparison based on the results available in the literature.

The second group consists of exact algorithms with publicly available source code, enabling more reproducible, controlled computational evaluation. These methods can be further divided into two main classes. The first class includes SAT-based approaches, such as \texttt{ZykovColor} and \texttt{Assignment} from \citet{brand2026customized}, as well as \texttt{gc-cdcl}~\citep{hebrard2020constraint}, \texttt{CliColCom}~\citep{heule2022cliques}, and \texttt{POP-S}~\citep{faber2024sat}. The second class includes \texttt{exactcolors}, the BP algorithm proposed by \citet{Held2012}, and \texttt{DSATUR}, the DSatur-based exact algorithm of \citet{san2012new}. Together, these algorithms represent some of the strongest publicly available exact methods for the GCP and cover a broad spectrum of modeling and search paradigms, including BP, SAT encodings, and DSatur-based branch-and-bound. 
All these publicly available algorithms are executed under the computational environment described in Section~\ref{sec:compenv}, thereby enabling a controlled assessment of the computational performance of \BPname. In addition, on the selected 500 Erd\H{o}s--R\'enyi benchmark instances, \BPname\ is compared with the same set of publicly available algorithms under this common experimental setting.

In the following, Sections~\ref{subsubsec:comparison_no_code} and~\ref{subsubsec:comparison_with_code} compare \BPname\ with state-of-the-art exact algorithms, distinguishing between approaches without and with publicly available source code.

\subsection{Comparison with Approaches without Publicly Available Source Code}
\label{subsubsec:comparison_no_code}

An aggregate comparison with \BPname\ is summarized below. However, differences in the tested DIMACS subsets and time limits preclude a comprehensive assessment.

\begin{itemize}
\item The best method reported by \citet{van2022graph} solves 50 instances from the full set of 137 DIMACS instances within a 1-hour time limit. On the same 137-instance benchmark set, \BPname\ solves all these 50 instances and 46 additional DIMACS instances, for a total of 96 solved instances under the same time limit.

\item \citet{malaguti2011exact} reports solving 65 instances on a subset of 109 DIMACS instances within a 10-hour time limit. On the same 109-instance subset, \BPname\ solves all 65 instances and 10 additional instances within 1 hour, for a total of 75 instances solved on that subset.

\item \citet{Gualandi2012} reports solving 10 instances on a subset of 17 DIMACS instances within a 10-hour time limit. On the same 17-instance subset, \BPname\ solves all 10 instances and 4 additional instances within 1 hour, for a total of 14 instances solved on that subset.

\item \citet{Morrison2016} reports solving 15 instances on a subset of 47 DIMACS instances within a 10-hour time limit. On the same 47-instance subset, \BPname\ solves 26 instances within 1 hour. This includes 14 of the 15 instances solved by \citet{Morrison2016}; the only instance solved by \citet{Morrison2016} but not by \BPname\ is \texttt{flat300\_28\_0}.
\end{itemize}

In summary, although the comparisons are not fully controlled due to differences in machines, time limits, and benchmark subsets, the available results indicate that \BPname\ is highly competitive with, and often substantially stronger than, previous exact BP and DD-based approaches on the reported DIMACS subsets. \BPname\ consistently solves substantially more instances, confirming its competitiveness as a state-of-the-art exact algorithm for the GCP. It is also worth noting that \BPname\ significantly improves upon previous BP algorithms. Although based on the same algorithmic framework, it is substantially stronger than the methods of \citet{Gualandi2012} and \citet{malaguti2011exact}, solving more instances in much shorter time. \BPname\ also surpasses the BP approach with decision diagrams proposed by \citet{Morrison2016} in both the number of solved instances and computational efficiency. 

\subsection{Comparison with Approaches with Publicly Available Source Code} \label{subsubsec:comparison_with_code}

The group of publicly available exact algorithms enables a more detailed and reproducible comparison. In particular, the recent study of \citet{brand2026customized} introduced two exact methods for the GCP, \texttt{ZykovColor} and \texttt{Assignment}. Together with \texttt{exactcolors}~\citep{Held2012}, \texttt{DSATUR}~\citep{san2012new}, \texttt{gc-cdcl}~\citep{hebrard2020constraint}, \texttt{CliColCom}~\citep{heule2022cliques}, and \texttt{POP-S}~\citep{faber2024sat}, these methods form the benchmark algorithms used in the following reproducible comparison. The corresponding source-code repositories are reported in the bibliographic entries of the cited papers. The computational study is carried out on the DIMACS instances and the selected Erd\H{o}s--R\'enyi random instances.

\subsubsection{Comparison on DIMACS Instances.}\label{subsec:DIMACS_op}
Table~\ref{tableForFamily} presents a detailed comparison of \BPname\ against the seven exact algorithms discussed above, evaluated on the full set of 137 DIMACS instances grouped by families. The first two columns report the family names and the corresponding number of instances (``\#inst''). The next 18 columns show, for each algorithm and each family, the number of instances solved to optimality (``\#opt'') and the average running time in seconds (``time''), computed over the solved instances. All algorithms were run on the same machine used for \BPname, with a time limit of 1 hour. 
Detailed computational results for the 137 DIMACS instances are reported in the e-companion; see \S\ref{subsec:detailed_comparison_with_code}.

\BPname\ is tested in two variants: one without using any upper bound as input, and one that incorporates the best known upper bound from the literature. The former allows a fair comparison with all other exact algorithms, which are fully self-contained and include their own internal heuristic subroutines. The latter, referred to as \BPUB, measures the performance of our exact method when assisted by the best available upper bound. Both variants are run with a 1-hour time limit. When the supplied upper bound equals the optimal chromatic number, the reported time corresponds to the time required to certify optimality.

Since our goal is to highlight performance, especially on hard instances of the GCP, all values below 0.01 seconds are reported in the table as ``\(\le 0.01\)''.
A dash (``-'') is reported in the table when none of the instances in a family are solved by a given algorithm. To facilitate comparison, the best results are highlighted in bold, prioritizing the number of solved instances and, in case of ties, the average running time. \BPUB\ is excluded from the boldface comparison because it uses external best-known upper bounds.

\begin{table}[]
\caption{Computational results of \BPname, \BPUB, and state-of-the-art exact GCP algorithms with available source code on the 137 DIMACS instances.}

  \centering
  \label{tableForFamily}
    \scalebox{0.75}{
      \footnotesize
      \renewcommand{\arraystretch}{1.5}
      \setlength{\tabcolsep}{3.0pt}
      \begin{tabular}{lr c *{9}{r r}}
    \toprule
    \multicolumn{1}{l}{\multirow{2}{*}{}} &
    \multicolumn{1}{c}{\multirow{2}{*}{ }} &
    \multicolumn{2}{c}{{\tt exactcolors}} &
    \multicolumn{2}{c}{{\tt DSATUR}} &
    \multicolumn{2}{c}{{\tt gc-cdcl}} &
    \multicolumn{2}{c}{{\tt CliColCom}} &
    \multicolumn{2}{c}{{\tt POP-S}} &
    \multicolumn{2}{c}{{\tt Assignment}} &
    \multicolumn{2}{c}{{\tt ZykovColor}} &
    \multicolumn{2}{c}{{\BPname}} &
    \multicolumn{2}{c}{\BPUB} \\
    \cmidrule(lr){3-4} \cmidrule(lr){5-6} \cmidrule(lr){7-8} \cmidrule(lr){9-10}
    \cmidrule(lr){11-12} \cmidrule(lr){13-14} \cmidrule(lr){15-16} \cmidrule(lr){17-18} \cmidrule(lr){19-20}
    Family & \#inst & \#opt & time & \#opt & time & \#opt & time & \#opt & time & \#opt & time & \#opt & time & \#opt & time & \#opt & time & \#opt & time \\
    \cmidrule(lr){1-2}
    \cmidrule(lr){3-4} \cmidrule(lr){5-6} \cmidrule(lr){7-8} \cmidrule(lr){9-10}
    \cmidrule(lr){11-12} \cmidrule(lr){13-14} \cmidrule(lr){15-16} \cmidrule(lr){17-18} \cmidrule(lr){19-20}
\texttt{C} & 3 & 0 & - & 0 & - & 0 & - & 0 & - & 0 & - & 0 & - & 0 & - & 0 & - & 0 & - \\

\texttt{DSJC} & 12 & 2 & 818.62 & 1 & $\le 0.01$ & 1 & 0.72 & 1 & 0.03 & 1 & 0.20 & 1 & 1.31 & 2 & 341.62 & \textbf{4} & 60.85 & 4 & 45.31 \\

\texttt{DSJR} & 3 & 3 & 1368.59 & 1 & 0.14 & 1 & 0.07 & 3 & 22.08 & 1 & 0.93 & 2 & 1.36 & 3 & 30.12 & \textbf{3} & 12.50 & 3 & 0.11 \\

\texttt{FullIns} & 14 & 5 & $\le 0.01$ & 12 & 0.20 & 13 & 12.00 & 14 & 46.00 & 14 & 8.23 & 14 & 4.57 & \textbf{14} & 3.06 & 11 & 1.70 & 11 & 0.15 \\

\texttt{GPIA} & 4 & 0 & - & 2 & 2.60 & 3 & 115.46 & 3 & 0.30 & \textbf{4} & 7.24 & 4 & 8.68 & 4 & 393.87 & 2 & 1301.81 & 2 & 1202.79 \\

\texttt{Insertions} & 11 & 1 & 587.41 & 3 & 799.53 & 3 & 1006.42 & 4 & 0.37 & 4 & 0.37 & \textbf{4} & 0.21 & 4 & 36.82 & 1 & 5.24 & 1 & 1.97 \\

\texttt{SocialNet} & 5 & 5 & 0.02 & 5 & $\le 0.01$ & 5 & $\le 0.01$ & 5 & 0.02 & 5 & 0.24 & 5 & 0.06 & 5 & 0.06 & \textbf{5} & $\le 0.01$ & 5 & $\le 0.01$ \\

\texttt{flat} & 6 & 2 & 714.64 & 0 & - & 0 & - & 0 & - & 0 & - & 0 & - & 0 & - & \textbf{2} & 261.49 & 3 & 157.24 \\

\texttt{fpsol} & 3 & 3 & 0.38 & 3 & $\le 0.01$ & 3 & 2.22 & 3 & 0.06 & 3 & 8.16 & 3 & 0.07 & 3 & 0.07 & \textbf{3} & $\le 0.01$ & 3 & $\le 0.01$ \\

\texttt{inithx} & 3 & 3 & 0.71 & 3 & 0.08 & 3 & 2.36 & 3 & 0.11 & 3 & 14.33 & 3 & 0.08 & 3 & 0.08 & \textbf{3} & $\le 0.01$ & 3 & $\le 0.01$ \\

\texttt{le} & 12 & 3 & 878.58 & 4 & 3.21 & 8 & 309.01 & 10 & 9.65 & 8 & 2.00 & 8 & 7.57 & 10 & 112.24 & \textbf{12} & 6.77 & 12 & $\le 0.01$ \\

\texttt{miles} & 5 & 5 & 0.08 & 5 & $\le 0.01$ & 5 & 0.53 & 5 & 0.03 & 5 & 4.65 & 5 & 0.06 & 5 & 0.06 & \textbf{5} & $\le 0.01$ & 5 & $\le 0.01$ \\

\texttt{mug} & 4 & 4 & 0.87 & 2 & 356.65 & 4 & 0.08 & \textbf{4} & 0.03 & 4 & 0.21 & 4 & 0.45 & 4 & 0.66 & 4 & 9.74 & 4 & 0.67 \\

\texttt{mulsol} & 5 & 5 & 0.09 & 5 & $\le 0.01$ & 5 & 0.42 & 5 & 0.04 & 5 & 3.94 & 5 & 0.06 & 5 & 0.07 & \textbf{5} & $\le 0.01$ & 5 & $\le 0.01$ \\

\texttt{myciel} & 5 & 2 & 9.57 & 3 & 0.48 & 5 & 0.02 & 4 & 206.43 & 4 & 383.06 & 5 & 0.06 & 5 & 0.06 & \textbf{5} & $\le 0.01$ & 5 & $\le 0.01$ \\

\texttt{others} & 3 & 2 & 0.74 & 2 & 0.12 & 2 & 0.24 & \textbf{2} & 0.11 & 2 & 1.36 & 2 & 0.40 & 2 & 0.69 & 2 & 55.19 & 2 & 5.68 \\

\texttt{qg.order} & 4 & 0 & - & 1 & 0.67 & 3 & 594.41 & 3 & 73.84 & 3 & 1066.31 & 4 & 521.33 & 3 & 1186.36 & \textbf{4} & 20.05 & 4 & $\le 0.01$ \\

\texttt{queen} & 13 & 7 & 35.15 & 6 & 106.15 & 6 & 35.06 & 5 & 2.26 & 8 & 463.91 & 6 & 80.80 & 8 & 38.55 & \textbf{9} & 329.23 & 12 & 4.35 \\

\texttt{r} & 9 & 7 & 42.21 & 7 & 35.03 & 7 & 5.72 & 8 & 417.27 & 7 & 27.46 & 8 & 155.66 & 7 & 14.09 & \textbf{9} & 471.45 & 9 & 145.93 \\

\texttt{school} & 2 & 2 & 1618.10 & 2 & 6.21 & 2 & 2.24 & 2 & 0.33 & 2 & 8.74 & 2 & 0.15 & 2 & 0.15 & \textbf{2} & $\le 0.01$ & 2 & $\le 0.01$ \\

\texttt{wap} & 8 & 1 & 5.49 & 1 & 0.91 & 1 & 4.29 & \textbf{4} & 116.98 & 4 & 1412.76 & 3 & 1322.23 & 2 & 226.27 & 2 & 2.85 & 5 & $\le 0.01$ \\

\texttt{zeroin} & 3 & 3 & 0.04 & 3 & $\le 0.01$ & 3 & 0.42 & 3 & 0.03 & 3 & 4.02 & 3 & 0.06 & 3 & 0.07 & \textbf{3} & $\le 0.01$ & 3 & $\le 0.01$ \\
\cmidrule(lr){1-2}
\cmidrule(lr){3-4} \cmidrule(lr){5-6} \cmidrule(lr){7-8} \cmidrule(lr){9-10}
\cmidrule(lr){11-12} \cmidrule(lr){13-14} \cmidrule(lr){15-16} \cmidrule(lr){17-18} \cmidrule(lr){19-20}
Total & 137 & 65 & & 71 & & 83 & & 91 & & 90 & & 91 & & 94 & & \textbf{96} & & 103 & \\
\bottomrule
\end{tabular}
    }
\end{table}

According to Table~\ref{tableForFamily}, 
among methods that do not use externally supplied upper bounds,
\BPname\ achieves the strongest overall performance, solving 96 instances. \BPUB\ solves 103 instances. In comparison, \texttt{ZykovColor} solves two fewer instances, while \texttt{Assignment}, \texttt{CliColCom}, and \texttt{POP-S} solve five fewer, \texttt{gc-cdcl} solves 13 fewer, \texttt{DSATUR} solves 25 fewer, and \texttt{exactcolors} solves 31 fewer instances. It is worth noting that, for \texttt{POP-S}, the results reported in the literature correspond to 91 solved instances, whereas on our machine we obtain 90. The difference is due to instance \texttt{wap02a}, which is solved in the literature results but not in our rerun, likely because of randomization in \texttt{POP-S}. For consistency with the reference study, the comparison below is therefore based on the literature value of 91 solved instances. It is worth noting that the only other BP algorithm considered, {\tt exactcolors}, demonstrates significantly lower computational performance than \BPname, despite using the same algorithmic paradigm. These results are consistent with those presented in the previous section and confirm that \BPname\ is by far the most effective BP algorithm for the GCP. While all the other approaches are based on different frameworks, most of which heavily rely on SAT solvers, \BPname\ still achieves highly competitive performance.

As far as the number of solved instances is concerned, \BPname\ outperforms all other approaches on several families of instances. 
\BPname\ is the only algorithm able to solve 4 out of the 12 instances in the {\tt DSJC} family (all other algorithms solve at most 2). 
In the {\tt le} family, \BPname\ solves all 12 instances, while the second-best algorithm, {\tt CliColCom}, solves only 10. In case of ties in the number of instances solved, we consider the second-best algorithm to be the one with the lowest average running time. 
In the {\tt queen} family, \BPname\ solves 9 out of 13 instances, outperforming the second-best algorithm, {\tt ZykovColor}, which solves 8.
On the other hand, \BPname\ is outperformed on the following families: in the {\tt FullIns} family, the best-performing algorithm is {\tt ZykovColor}, which solves all 14 instances, whereas \BPname\ solves only 11. In the {\tt GPIA} family, {\tt POP-S} solves all 4 instances, while \BPname\ solves only 2. 
For the {\tt Insertions} family, {\tt Assignment} performs best, solving 4 out of 11 instances, compared to only 1 by \BPname. The performance of \BPname\ is largely complementary to that of existing approaches, achieving state-of-the-art results on some instance families while being outperformed on others.

Regarding the efficiency on solved instances, \BPname\ demonstrates superior performance on several families. In the \texttt{DSJR} family, it achieves an average running time of 12.5 seconds, significantly faster than the second-best method \texttt{CliColCom}, which also solves all instances but requires 22.08 seconds. In the \texttt{qg.order} family, \BPname\ solves all instances in 20.05 seconds, outperforming \texttt{Assignment} by a large margin, as the latter requires 521.33 seconds. Similarly, in the \texttt{school} family, \BPname\ completes all instances in less than 0.01 seconds, while the second-best method, again \texttt{Assignment}, takes 0.15 seconds. On the other hand, in families such as \texttt{SocialNet}, \texttt{fpsol}, \texttt{inithx}, \texttt{miles}, \texttt{mulsol} and \texttt{zeroin}, \BPname\ solves all instances within 0.01 seconds, consistent with most of the other algorithms, indicating these families are relatively easier across methods.

When best-known upper bounds are provided to \BPname, \BPUB\ demonstrates superior performance in terms of both the number of solved instances and average running time. It successfully solves 7 additional instances that \BPname\ fails to solve, particularly from the \texttt{flat}, \texttt{queen}, and \texttt{wap} families. Moreover, among the 96 instances solved by both methods, \BPUB\ significantly reduces the average runtime across several families, with notable improvements observed in \texttt{DSJR} (from 12.5s to 0.11s), \texttt{le} (from 6.77s to $\le$0.01s), \texttt{others} (from 55.19s to 5.68s), and \texttt{qg.order} (from 20.05s to $\le$0.01s). 

\begin{figure}[t]
\caption{Number of DIMACS instances solved over time by \BPname\ and seven state-of-the-art exact GCP algorithms (137 instances, 1 hour time limit).}
\label{fig:survivalPlot_DIMACS}
\center
\includegraphics[width=0.95\columnwidth]{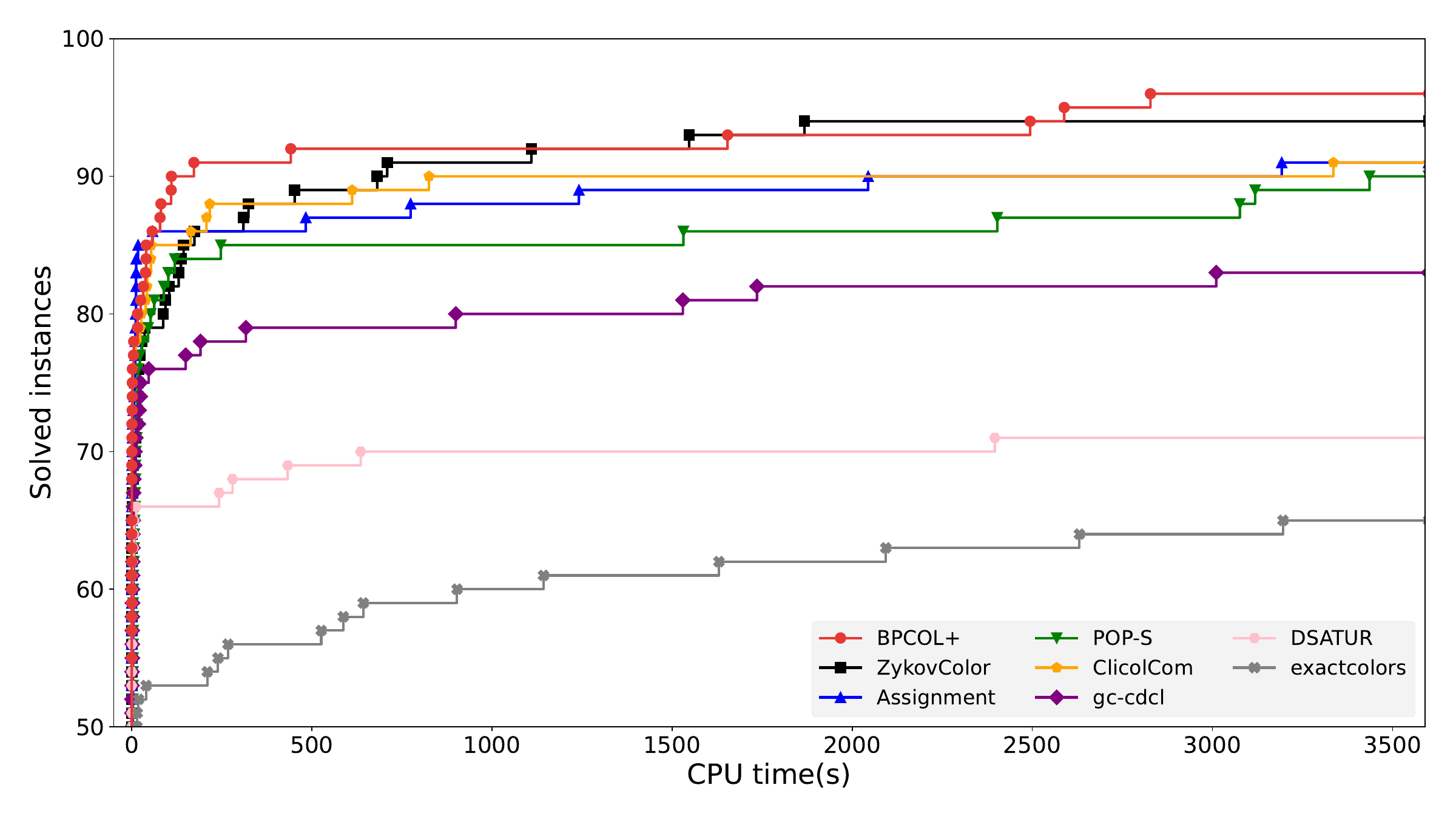}
\end{figure}

Figure~\ref{fig:survivalPlot_DIMACS} shows the number of instances solved as a function of computing time to visually compare the computational behavior of the eight algorithms over the full set of 137 DIMACS instances. The horizontal axis reports the elapsed running time, while the vertical axis shows the cumulative number of instances solved by each algorithm up to that time. Thus, each curve increases as more instances are solved. A curve that rises more rapidly indicates that the algorithm solves instances faster, whereas a higher final plateau indicates that more instances are solved within the prescribed time limit.
The figure shows that \BPname\ exhibits the best overall behavior. Its curve rises very rapidly at the beginning, indicating that many easy and medium-difficulty instances are solved within a short time. It reaches about 90 solved instances within the first 100 seconds and continues to improve afterward, eventually attaining a final plateau of 96. \texttt{ZykovColor} is the closest competitor: it also shows strong early performance and reaches a high final plateau, but remains below \BPname\ over most of the time horizon. \texttt{CliColCom}, \texttt{Assignment}, and \texttt{POP-S} form a second group of competitive methods. They solve many instances quickly at the beginning, but their curves plateau earlier and at lower levels. 
The remaining methods are less competitive on this benchmark set. \texttt{gc-cdcl} reaches a final plateau of 83 solved instances, while \texttt{DSATUR} and \texttt{exactcolors} solve substantially fewer instances, with final plateaus of 71 and 65 instances, respectively. Overall, the results indicate that \BPname\ is both efficient and robust, achieving the highest number of solved instances while maintaining strong early performance.

\subsubsection{Comparison on Erd\H{o}s--R\'enyi Instances.}

\begin{table}[t]
\caption{Computational results of \BPname and state-of-the-art exact GCP algorithms with available source code on the 500 Erd\H{o}s--R\'enyi instances grouped by density and by number of vertices.}
\centering
\label{tableForDensityClass}
\scalebox{0.75}{
\footnotesize
\renewcommand{\arraystretch}{1.5}
\setlength{\tabcolsep}{5.0pt}
\begin{tabular}{l c r *{8}{r r}}
\toprule
\multicolumn{1}{c}{\multirow{2}{*}{ }} &
\multicolumn{1}{c}{\multirow{2}{*}{ }} &
\multicolumn{1}{c}{\multirow{2}{*}{ }} &
\multicolumn{2}{c}{\texttt{exactcolors}} &
\multicolumn{2}{c}{\texttt{DSATUR}} &
\multicolumn{2}{c}{\texttt{gc-cdcl}} &
\multicolumn{2}{c}{\texttt{CliColCom}} &
\multicolumn{2}{c}{\texttt{POP-S}} &
\multicolumn{2}{c}{\texttt{Assignment}} &
\multicolumn{2}{c}{\texttt{ZykovColor}} &
\multicolumn{2}{c}{\texttt{\BPname}} \\
\cmidrule(lr){4-5} \cmidrule(lr){6-7} \cmidrule(lr){8-9} \cmidrule(lr){10-11}
\cmidrule(lr){12-13} \cmidrule(lr){14-15} \cmidrule(lr){16-17} \cmidrule(lr){18-19}
Group &  $\mu$ ($\%)$ & \#inst & \#opt & time & \#opt & time & \#opt & time & \#opt & time & \#opt & time & \#opt & time & \#opt & time & \#opt & time \\
\cmidrule(lr){1-3}
\cmidrule(lr){4-5} \cmidrule(lr){6-7} \cmidrule(lr){8-9} \cmidrule(lr){10-11}
\cmidrule(lr){12-13} \cmidrule(lr){14-15} \cmidrule(lr){16-17} \cmidrule(lr){18-19}
\texttt{ER*.05} & 5 & 50 & 49 & 23.70 & \textbf{50} & $\le 0.01$ & 50 & 0.03 & 50 & 0.09 & 50 & 0.33 & 50 & 0.24 & 50 & 0.24 & 50 & 67.20 \\
\texttt{ER*.1} & 10 & 50 & 49 & 130.54 & \textbf{50} &$\le 0.01$& 50 & 0.59 & 50 & 0.09 & 50 & 0.35 & 50 & 0.47 & 50 & 0.72 & 49 & 204.93 \\
\texttt{ER*.15} & 15 & 50 & 43 & 310.04 & 50 & 9.08 & 46 & 45.23 & 50 & 38.41 & \textbf{50} & 4.99 & 50 & 6.02 & 50 & 65.98 & 48 & 270.39 \\
\texttt{ER*.2} & 20 & 50 & 35 & 380.06 & \textbf{49} & 273.43 & 36 & 112.50 & 42 & 51.48 & 45 & 135.93 & 45 & 67.21 & 42 & 133.59 & 38 & 108.36 \\
\texttt{ER*.25} & 25 & 50 & 34 & 422.44 & \textbf{47} & 322.59 & 29 & 610.69 & 38 & 267.12 & 40 & 60.92 & 42 & 207.69 & 38 & 191.70 & 44 & 112.14 \\
\texttt{ER*.3} & 30 & 50 & 29 & 390.35 & 32 & 130.32 & 18 & 198.98 & 31 & 199.62 & 36 & 265.77 & 34 & 180.48 & 29 & 193.65 & \textbf{43} & 228.28 \\
\texttt{ER*.4} & 40 & 50 & 33 & 274.93 & 29 & 124.48 & 13 & 1807.30 & 20 & 319.72 & 28 & 377.05 & 27 & 73.58 & 25 & 377.49 & \textbf{43} & 56.58 \\
\texttt{ER*.5} & 50 & 50 & 37 & 497.88 & 22 & 348.25 & 6 & 1085.35 & 12 & 316.25 & 12 & 180.69 & 15 & 846.04 & 11 & 239.25 & \textbf{46} & 88.75 \\
\texttt{ER*.7} & 70 & 50 & 49 & 213.03 & 20 & 302.64 & 2 & 2329.52 & 13 & 616.80 & 12 & 461.94 & 19 & 135.33 & 10 & 1020.57 & \textbf{50} & 19.39 \\
\texttt{ER*.9} & 90 & 50 & \textbf{50} & 1.95 & 38 & 151.02 & 15 & 510.71 & 32 & 418.21 & 28 & 289.65 & 48 & 138.42 & 50 & 18.73 & 50 & 8.81 \\
\midrule
Group & $|\mathcal{V}|$ & \#inst & \#opt & time & \#opt & time & \#opt & time & \#opt & time & \#opt & time & \#opt & time & \#opt & time & \#opt & time \\
\cmidrule(lr){1-3}
\cmidrule(lr){4-5} \cmidrule(lr){6-7} \cmidrule(lr){8-9} \cmidrule(lr){10-11}
\cmidrule(lr){12-13} \cmidrule(lr){14-15} \cmidrule(lr){16-17} \cmidrule(lr){18-19}
\texttt{ER70} & 70 & 100 & 100 & 17.45 & \textbf{100} & 1.82 & 87 & 401.55 & 100 & 59.76 & 99 & 20.40 & 99 & 6.57 & 99 & 115.24 & 100 & 8.34 \\
\texttt{ER80} & 80 & 100 & 93 & 260.41 & 100 & 140.71 & 67 & 261.93 & 83 & 144.08 & 83 & 105.99 & 89 & 88.96 & 82 & 83.87 & \textbf{100} & 103.96 \\
\texttt{ER90} & 90 & 100 & 91 & 205.86 & 79 & 63.29 & 46 & 243.13 & 68 & 183.85 & 74 & 237.93 & 82 & 130.97 & 74 & 150.46 & \textbf{97} & 194.09 \\
\texttt{ER100} & 100 & 100 & 74 & 358.20 & 60 & 187.92 & 39 & 110.42 & 51 & 187.65 & 58 & 140.97 & 64 & 149.25 & 57 & 123.49 & \textbf{93} & 114.97 \\
\texttt{ER110} & 110 & 100 & 50 & 534.04 & 48 & 535.81 & 26 & 69.44 & 36 & 333.74 & 37 & 221.19 & 46 & 288.14 & 43 & 200.25 & \textbf{71} & 175.00 \\
\cmidrule(lr){1-3}
\cmidrule(lr){4-5} \cmidrule(lr){6-7} \cmidrule(lr){8-9} \cmidrule(lr){10-11}
\cmidrule(lr){12-13} \cmidrule(lr){14-15} \cmidrule(lr){16-17} \cmidrule(lr){18-19}
Total &  & 500 & 408 &  & 387 &  & 265 &  & 338 &  & 351 &  & 380 &  & 355 &  & \textbf{461} &  \\
\bottomrule

\end{tabular}
}
\end{table}

\begin{figure}[t]
\caption{Number of Erd\H{o}s--R\'enyi instances solved over time by \BPname\ and seven exact GCP algorithms (500 instances, 1 hour time limit).}
\label{fig:survivalPlotER}
\center
\includegraphics[width=0.95\columnwidth]{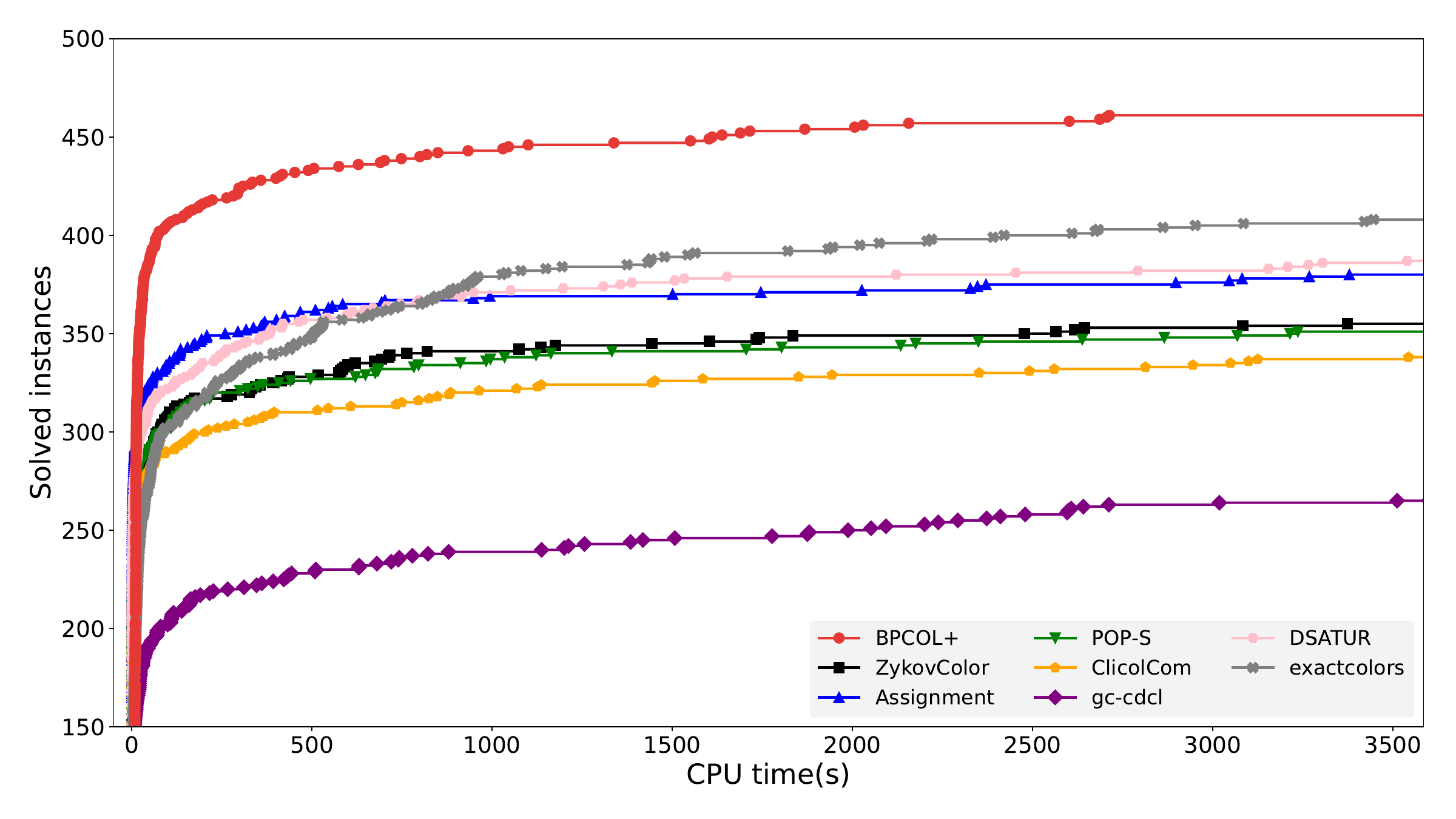}
\end{figure}

For a controlled comparison with publicly available exact algorithms, we use the previously selected subset of 500 Erd\H{o}s--R\'enyi instances. Table~\ref{tableForDensityClass} and Figure~\ref{fig:survivalPlotER} summarize the performance of \BPname\ and the seven publicly available exact algorithms on the selected  Erd\H{o}s--R\'enyi instances. Table~\ref{tableForDensityClass} reports aggregated results, grouped first by graph density (``\(\mu\)'') and then by the number of vertices (``\(|\mathcal{V}|\)''). The number of instances in each group is reported in column ``\#inst''. For each algorithm, column ``\#opt'' denotes the number of instances solved to proven optimality within the one-hour time limit, and column ``time'' denotes the average running time in seconds over the instances solved by that algorithm in the corresponding group. Figure~\ref{fig:survivalPlotER} complements the table by showing the cumulative number of solved instances over CPU time. 

Overall, \BPname\ achieves the best performance in terms of the total number of solved instances. It solves 461 out of 500 instances, substantially more than the closest competitors, namely \texttt{exactcolors} (408 solved instances) and \texttt{DSATUR} (387 solved instances). This advantage is clearly visible in Figure~\ref{fig:survivalPlotER}: the curve of \BPname\ rises sharply at the beginning and remains above all other curves throughout almost the entire time horizon, eventually reaching the highest plateau. This indicates that \BPname\ is not only robust in terms of the number of solved instances, but also competitive in terms of time-to-solution.

The aggregated results further show that the advantage of \BPname\ is most pronounced on larger and higher-density instances. In terms of graph density, several algorithms perform well on very sparse instances, such as \texttt{ER*.05} and \texttt{ER*.1}, where almost all methods solve nearly all instances. On these low-density groups, \BPname\ does not show a clear advantage. However, as density increases, \BPname\ becomes clearly more competitive. It solves the largest number of instances for densities \(30\%\), \(40\%\), \(50\%\), and \(70\%\), and ties for the best result at density \(90\%\). This suggests that the proposed method is particularly effective on denser random graphs, where several SAT-based and Dsatur-based methods exhibit a substantial drop in the number of solved instances.
A similar pattern is observed for the number of vertices. While all methods are relatively competitive on smaller graphs with 70 vertices, the gap widens as the graph size increases. For \(n=90\), \(n=100\), and \(n=110\), \BPname\ solves 97, 93, and 71 instances, respectively, achieving the best result in each of these groups. These results indicate that \BPname\ scales better than the other algorithms on larger Erd\H{o}s--R\'enyi instances, especially when combined with higher graph densities.

We also evaluated \BPname\ on the full set of 5,000 Erd\H{o}s--R\'enyi instances. Within the one-hour time limit, \BPname\ solves 4,641 instances to proven optimality, leaving 359 instances unsolved. The phase-by-phase progression of these instances is discussed in the next section and summarized in Table~\ref{tab:phase_summary}.

\subsection{Effectiveness of the Main Phases of \BPname}\label{subsec:effect_components}

This section summarizes the contribution of the main phases of \BPname. We focus the discussion on the DIMACS benchmark instances, for which we provide a detailed phase-by-phase analysis of preprocessing, heuristic initialization, root-node column generation, and ZDD-based branch-and-price. For each phase, we report the numbers of input instances, optimally solved instances, and remaining instances. Table~\ref{tab:phase_summary} reports the corresponding phase-level statistics for both the DIMACS and Erd\H{o}s--R\'enyi benchmark instances.
Additional experiments on the effectiveness of the main algorithmic components are reported in the e-companion; see \S\ref{subsec:component_effectiveness}.
Detailed results are available in the online repository.

\begin{itemize}
    \item In the preprocessing phase (\S\ref{subsec:preprocessing}), vertex reduction is substantial. Among the 137 DIMACS instances, 77 exhibit a decrease in the number of vertices. For these instances, the average reduction is nearly \(60\%\), and the maximum reduction reaches \(100\%\). In particular, more than half of the vertices are eliminated in 46 instances. The preprocessing phase is also very fast: in most cases, its running time essentially coincides with that of \texttt{CliSAT}. The latter computes a strong clique within the 1-second time limit on 123 instances, with an average running time of 0.19 seconds; on the remaining 14 very large instances, it reaches the time limit while still providing high-quality lower bounds. The greedy heuristic used to compute the Mycielski-based bound requires, on average, 0.02 seconds. Over the testbed, the clique bound is stronger on 87 instances, the Mycielski-based bound is stronger on 11 instances, and the two bounds coincide on 39 instances. Preprocessing alone certifies optimality in 22 instances, leaving 115 for later phases.
    
    \item The heuristic initialization phase (\S\ref{subsec:primal}) also proves effective on the remaining 115 DIMACS instances. In this phase, \BPname\ applies the first phase of the MMT heuristic to compute high-quality initial colorings and upper bounds. The heuristic takes about 8 seconds on average when the time limit is not reached, and exceeds it in 30 cases. It matches the best-known upper bound from the literature in 87 instances. In the remaining 28 instances, the obtained upper bound is worse than the best-known one; nevertheless, several of these instances are eventually solved to optimality by the subsequent phases of \BPname. By combining the heuristic upper bounds with the clique-based lower bounds obtained in preprocessing, 43 additional instances are solved to proven optimality before entering the main BP process. Consequently, 72 instances remain unsolved after preprocessing and heuristic initialization.
    
    \item Root-node column generation (\S\ref{subsec:safe_bounds} and \S\ref{subsec:lb_zdd}) is then applied to these 72 remaining DIMACS instances. It solves the root node and computes a valid lower bound within the one-hour time limit for 52 instances, with an average running time of about 90 seconds. Among these 52 instances, 15 are solved in less than 0.1 seconds, 36 under 10 seconds, and only 3 require more than 10 minutes. In 16 cases, the root-node lower bound is sufficient to certify optimality. The other 36 instances whose root nodes are solved but not proven optimal are passed to the ZDD construction and BP phases. In the remaining 20 cases, the root node is not solved within the time limit; the main difficulty stems from the graph's sparsity, which makes exact pricing particularly challenging. For the 52 instances whose root node is solved, about \(99\%\) of the generated columns are produced by the heuristic pricer, which accounts for roughly \(15\%\) of the total time; exact pricing accounts for about \(35\%\), and solving the RMP accounts for another \(45\%\). The number of generated columns ranges from a few to more than 10,000, with an average of about 1,000 per instance.
    
    \item After root-node column generation, ZDD construction (\S\ref{subsec:branching} and \S\ref{sec:ZDDconstr}) is attempted on the remaining 36 DIMACS instances that are not proven optimal at the root node. In terms of ZDD construction, a complete ZDD is successfully produced for 13 instances, and a reduced ZDD for 18 instances; in 5 cases, neither construction completes within the one-hour time limit. For reduced ZDDs, the average proportion of eliminated maximal stable sets is about \(45\%\), and the average construction time is about 60 seconds, whereas complete ZDD construction requires about 130 seconds on average. With ZDD-based pricing, the subsequent BP procedure solves 15 of these 36 instances, while the remaining 21 reach the overall time limit. Restarting CG on the reduced ZDD improves the corresponding fractional lower bound in 10 cases, and in 5 of them, optimality is established directly from the updated bound.
\end{itemize}

\begin{table}[t]
\centering
\footnotesize
\caption{Progress of \BPname\ across the main phases of the algorithm on the DIMACS and Erd\H{o}s--R\'enyi instances.}
\label{tab:phase_summary}
\scalebox{1.0}{
\begin{tabular}{lrrrrrr}
\toprule
\multirow{2}{*}{Phase}
& \multicolumn{3}{c}{DIMACS}
& \multicolumn{3}{c}{Erd\H{o}s--R\'enyi} \\
\cmidrule(lr){2-4} \cmidrule(lr){5-7}
& \#Input & \#Solved & \#Remaining
& \#Input & \#Solved & \#Remaining \\
\midrule
Preprocessing phase (\S\ref{subsec:preprocessing})
& 137 & 22 & 115
& 5,000 & 46 & 4,954 \\

Heuristic initialization phase (\S\ref{subsec:primal})
& 115 & 43 & 72
& 4,954 & 336 & 4,618 \\

Root-node CG phase (\S\ref{subsec:safe_bounds} and \S\ref{subsec:lb_zdd})
& 72 & 16 & 56
& 4,618 & 2,471 & 2,147 \\

ZDD construction and BP phase (\S\ref{subsec:branching} and \S\ref{sec:ZDDconstr})
& 36 & 15 & 21
& 2,147 & 1,788 & 359 \\

\midrule
Total
& 137 & 96 & 41
& 5,000 & 4,641 & 359 \\
\bottomrule
\end{tabular}
}

\vspace{2mm}
\parbox{0.95\linewidth}{
}
\end{table}

Table~\ref{tab:phase_summary} summarizes the aforementioned progress through the main phases of \BPname\ for the DIMACS instances, reporting the number of input, solved, and remaining instances at each stage. The same analysis is also provided for the 5,000 Erd\H{o}s--R\'enyi instances, for which preprocessing solves 46 instances and leaves 4,954 instances for heuristic initialization. The heuristic initialization phase solves another 336 instances, leaving 4,618 instances for root-node column generation. The root-node CG phase is particularly effective on this benchmark, solving 2,471 additional instances and reducing the number of instances not yet proven optimal to 2,147. Among these, 2,129 instances successfully complete ZDD construction and proceed to the BP phase, while 18 instances remain unsolved because the ZDD cannot be constructed within the time limit. The ZDD-based BP phase then solves 1,788 of the 2,129 processed instances, leaving 341 unsolved. Consequently, \BPname\ solves \(46+336+2{,}471+1{,}788=4{,}641\) out of the 5,000 Erd\H{o}s--R\'enyi instances to proven optimality within the one-hour time limit, while \(18+341=359\) instances remain unsolved.

\section{Conclusions}\label{sec:conclusions}

This paper addressed the exact solution of the GCP and proposed \BPname, a state-of-the-art BP algorithm enhanced with preprocessing, heuristic initialization, and ZDD-based pricing. The main contribution is the integration of reduced-cost fixing with ZDD construction in a unified exact framework. Instead of building a complete ZDD of all maximal stable sets, \BPname\ exploits reduced-cost information to safely eliminate redundant ones during the construction of a reduced ZDD. We further introduce a new dual formulation that generates alternative dual solutions, strengthening reduced-cost fixing. Overall, these components significantly reduce the pricing space before branching while preserving exactness.

The computational results confirm the effectiveness of \BPname\ on both the classical DIMACS benchmark set and the recently proposed Erd\H{o}s--R\'enyi instances. On the 137 DIMACS instances, \BPname\ solves 96 instances to proven optimality within the prescribed time limit, a result that is highly competitive among state-of-the-art exact algorithms with publicly available source code and slightly improves upon the latest \texttt{ZykovColor} solver, which solves 94 instances on the same benchmark set. The comparison with algorithms whose source code is not publicly available also shows that \BPname\ is competitive with state-of-the-art exact approaches reported in the literature, even though these methods were often tested under different time limits and on different subsets of instances. For example, on the subset considered by \citet{Morrison2016}, \BPname\ solves 26 instances in one hour, whereas their method solves 15 in ten hours. The results on Erd\H{o}s--R\'enyi random graphs provide an additional validation of the method: on the full set of 5,000 instances, \BPname\ solves 4,641 instances within one hour; and in the controlled comparison with publicly available exact algorithms on the selected 500-instance subset, \BPname\ solves 461 instances, outperforming the closest competing methods, including \texttt{exactcolors} with 408 solved instances and \texttt{DSATUR} with 387 solved instances.
For the instances where both the complete ZDD and the reduced ZDD can be constructed, the reduced-cost-fixing strategy reduces the number of maximal stable sets by about \(50\%\) on average, with a maximum reduction exceeding \(99\%\). In aggregate, over these instances, the total number of maximal stable sets is reduced by about \(82\%\).
These results show that \BPname\ is robust and competitive across both DIMACS and randomly generated instances, especially on larger, denser graphs.

Beyond graph coloring, the proposed ZDD reduction mechanism may be useful for other optimization problems in which ZDDs represent large families of feasible structures, such as scheduling, packing, and related enumeration or pricing problems. In these settings, using dual information to reduce the represented solution space may improve the scalability of ZDD-based exact methods.

Several directions remain for future research. The ZDD-based pricing procedure could be strengthened by generating multiple useful columns in a single pricing call, thereby reducing the number of column-generation iterations. Another important direction is to accelerate the construction of reduced ZDDs, which can still be costly on large instances. Finally, the constructed ZDDs could be exploited more directly in mathematical programming formulations for the GCP and adapted to related coloring variants, such as weighted graph coloring, partition coloring, and other constrained coloring problems.

\bibliographystyle{informs2014}
\bibliography{myref}


%
%
%

\ECSwitch



\ECHead{E-Companion to the Paper Titled ``Advancing Branch-and-Price for Graph Coloring: New Pricing Strategies and Benchmark Results''}

\section{Proof of Statements} \label{sec:EC-Proof}

This section provides proofs of the statements reported in the main paper.

\subsection{Proof of Lemma \ref{lemma:sc_fix}}\label{sec:EC-ProofProp1}

The result can be established using arguments analogous to those in \citet{Agarwal1989}, \citet{baldacci2002new}, and \citet{Baldacci2008}. For completeness, we present a corresponding proof below.
\begin{proof}{Proof.}
Let \(\boldsymbol\xi\) be a feasible integer solution of~\eqref{SC} with value
\(z:=\sum_{\mathcal S\in\mathscr S}\xi_{\mathcal S}\le \tau\). By the definition of reduced cost in~\eqref{RC},
\begin{equation}
\sum_{\mathcal S\in\mathscr S} rc(\mathcal S,\boldsymbol\pi)\xi_{\mathcal S}
=
\sum_{\mathcal S\in\mathscr S}\xi_{\mathcal S}
-
\sum_{v\in\mathcal V}\pi_v
\sum_{\mathcal S\in\mathscr S(v)}\xi_{\mathcal S}.
\label{eq:scfix_sumrc}
\end{equation}
Since \(\boldsymbol\xi\) is feasible for~\eqref{SC}, we have
\(\sum_{\mathcal S\in\mathscr S(v)}\xi_{\mathcal S}\ge 1\) for every \(v\in\mathcal V\). Together with \(\boldsymbol\pi\ge\boldsymbol 0\), this implies
\[
\sum_{v\in\mathcal V}\pi_v
\sum_{\mathcal S\in\mathscr S(v)}\xi_{\mathcal S}
\ge
\sum_{v\in\mathcal V}\pi_v .
\]
Since \(z=\sum_{\mathcal S\in\mathscr S}\xi_{\mathcal S}\), \eqref{eq:scfix_sumrc} gives
\[
z
=
\sum_{\mathcal S\in\mathscr S}rc(\mathcal S,\boldsymbol\pi)\xi_{\mathcal S}
+
\sum_{v\in\mathcal V}\pi_v
\sum_{\mathcal S\in\mathscr S(v)}\xi_{\mathcal S}
\ge
\sum_{\mathcal S\in\mathscr S}rc(\mathcal S,\boldsymbol\pi)\xi_{\mathcal S}
+
\sum_{v\in\mathcal V}\pi_v .
\]
If \(\xi_{\bar{\mathcal S}}=1\), dual feasibility gives
\(rc(\mathcal S,\boldsymbol\pi)\ge0\) for every \(\mathcal S\in\mathscr S\), and hence
\[
z\ge rc(\bar{\mathcal S},\boldsymbol\pi)+\sum_{v\in\mathcal V}\pi_v>\tau,
\]
contradicting \(z\le \tau\).\Halmos
\end{proof}

\subsection{Proof of Lemma \ref{lemma:infeasible_fix}}
\label{sec:EC-ProofProp2}

\begin{proof}{Proof.}
Let \(\boldsymbol\xi\) be a feasible integer solution of~\eqref{SC} with value
\(z:=\sum_{\mathcal S\in\mathscr S}\xi_{\mathcal S}\le \tau\), and suppose by contradiction
that \(\xi_{\bar{\mathcal S}}=1\). As for the proof of Lemma~\ref{lemma:sc_fix}, we have
\[
z\ge
\sum_{\mathcal S\in\mathscr S}rc(\mathcal S,\boldsymbol\pi)\xi_{\mathcal S}
+\sum_{v\in\mathcal V}\pi_v .
\]
Since \(rc(\mathcal S,\boldsymbol\pi)\ge rc^*(\boldsymbol\pi)\) for every
\(\mathcal S\in\mathscr S\), we obtain
\[
z\ge
rc(\bar{\mathcal S},\boldsymbol\pi)+(z-1)rc^*(\boldsymbol\pi)
+\sum_{v\in\mathcal V}\pi_v .
\]
Because \(z\le \tau\) and \(rc^*(\boldsymbol\pi)<0\),
\((z-1)rc^*(\boldsymbol\pi)\ge (\tau-1)rc^*(\boldsymbol\pi)\). Hence
\[
z\ge
rc(\bar{\mathcal S},\boldsymbol\pi)+(\tau-1)rc^*(\boldsymbol\pi)
+\sum_{v\in\mathcal V}\pi_v>\tau,
\]
where the last inequality follows from
\eqref{eq:infeasible_fix_condition}. This contradicts \(z\le \tau\). \Halmos
\end{proof}

\subsection{Proof of Theorem~\ref{thm:pruning_rule}}
\label{sec:EC-ProofProp3}

\begin{proof}{Proof}
Let \(\hat{\mathcal S}\in\mathscr S\) be a maximal stable set generated below the current branch. Then
	\(\bar{\mathcal S}\subseteq\hat{\mathcal S}\subseteq\bar{\mathcal S}\cup\mathcal C\).
    Since \(\boldsymbol\pi_h\in\mathbb R^{|\mathcal V|}_{\ge0}\), condition~\eqref{eq:pruning_condition} implies
	\[
	rc(\hat{\mathcal S},\boldsymbol\pi_h)
	=1-\sum_{v\in\hat{\mathcal S}}\pi_{h,v}
	\ge
	1-\sum_{v\in\bar{\mathcal S}\cup\mathcal C}\pi_{h,v}
	>\delta_h .
	\]
	By the definition of \(\delta_h\) in~\eqref{eq:delta_def}, Lemma~\ref{lemma:sc_fix} applies if
	\(\boldsymbol\pi_h\in\bm\Pi_F\), and Lemma~\ref{lemma:infeasible_fix}
	applies if \(\boldsymbol\pi_h\in\bm\Pi_I\). Thus no feasible integer solution of~\eqref{SC} with objective value at most \(\tau\) can satisfy \(\xi_{\hat{\mathcal S}}=1\).
\Halmos
\end{proof} 

\subsection{Proof of Corollary~\ref{cor:reduced_zdd_correctness}}
\label{sec:EC-ProofCor1}

\begin{proof}{Proof.}
Consider first the recursive ZDD construction without reduced-cost pruning. By the correctness of the standard maximal-stable-set ZDD construction, the recursion generates exactly the root-to-\textsc{True} paths associated with maximal stable sets in \(\mathscr S\).

We now show that the additional reduced-cost tests remove exactly the maximal stable sets that do not belong to \(\mathscr S_{\rm red}\). At a terminal call, the current path corresponds to a maximal stable set \(\mathcal S\). The algorithm rejects this path if and only if there exists \(h\in\{1,2,\ldots,q\}\) such that
\[
rc(\mathcal S,\boldsymbol\pi_h)>\delta_h.
\]
Equivalently, the terminal test accepts exactly those complete paths satisfying
\[
rc(\mathcal S,\boldsymbol\pi_h)\le \delta_h,
\qquad h\in\{1,2,\ldots,q\},
\]
which are precisely the inequalities defining \(\mathscr S_{\rm red}\).

It remains to verify that early pruning cannot remove any element of \(\mathscr S_{\rm red}\). Suppose that a subtree is pruned by the test of Theorem~\ref{thm:pruning_rule}. Then, for every maximal stable set \(\hat{\mathcal S}\) represented by a completion of that subtree, there exists \(h\in\{1,2,\ldots,q\}\) such that
\[
rc(\hat{\mathcal S},\boldsymbol\pi_h)>\delta_h.
\]
Thus no such \(\hat{\mathcal S}\) belongs to \(\mathscr S_{\rm red}\). Conversely, if \(\mathcal S\in\mathscr S_{\rm red}\), then along the unique recursive path representing \(\mathcal S\), no reduced-cost pruning test can be satisfied; otherwise Theorem~\ref{thm:pruning_rule} would imply that \(\mathcal S\notin\mathscr S_{\rm red}\), a contradiction.

Finally, the node-merging and zero-suppression operations are standard ZDD reductions and preserve the represented family of sets. Therefore, the returned ZDD encodes exactly the family \(\mathscr S_{\rm red}\) defined in~\eqref{eq:sred_def}. \Halmos
\end{proof}

\section{Illustration of Complete and Reduced ZDD Construction}
\label{sec:app_zdd_illustration}

In this section, we present step-by-step examples of the complete and reduced ZDDs on a small, representative graph. These examples demonstrate how the recursive branching procedure builds the complete ZDD, and how the reduced-cost pruning selectively discards branches to yield a more compact reduced ZDD.

For the example of graph in Figure~\ref{fig:comparison_zdd}(a), Figure~\ref{fig:six-subfigs} illustrates the key steps for constructing the complete ZDD that encodes all maximal stable sets $\mathscr{S}$. Here, the notation $[n]$ denotes the set $\{1,2,\dots,n\}$ and a set listed in braces $\{\cdot\}$ indicates the uncovered set $\mathcal{U}$ for the corresponding recursive call in the $MakeIndSetZDD$ procedure of~\citet{Morrison2016}.

As shown in Figure~\ref{fig2:sub-1}, the construction starts from the root node associated with $v_1$, with the initial uncovered set $\mathcal{U}=[6]$. We denote by $\mathcal{R}$ the partial stable set induced by the vertices selected along the current root-to-node path. When the search takes the root's high branch (solid arc), vertex $v_1$ is added to $\mathcal{R}$, and $v_1$ together with its neighbors $(v_2,v_4,v_6)$ is removed from $\mathcal{U}$, yielding $\mathcal{U}=\{3,5\}$; the algorithm then proceeds to construct a node for $v_3$. If $v_3$ is selected, it is added to $\mathcal{R}$, and since $v_5$ is not adjacent to $v_3$, the uncovered set becomes $\mathcal{U}=\{5\}$. Selecting $v_5$ next yields $\mathcal{U}=\emptyset$, meaning that $\mathcal{R}=\{v_1,v_3,v_5\}$ is a maximal stable set; therefore, the high branch of $v_5$ is connected to the \textsc{True} node. Conversely, if $v_5$ is not selected, maximality cannot be satisfied, and the low branch of $v_5$ is connected to the \textsc{False} node. This completes the exploration of the high branch of $v_3$. We now backtrack and consider the low branch of $v_3$, as illustrated in Figure~\ref{fig2:sub-2}. At this point, $v_3$ is removed from the current partial stable set $\mathcal{R}$, so the algorithm restores $\mathcal{R}$ to its state before selecting $v_3$. If $v_3$ is not selected, the only selectable uncovered vertex in $\mathcal{U}$ is $v_5$, which is not adjacent to $v_3$. Therefore, this branch cannot lead to a maximal stable set, and the low branch of $v_3$ is pruned and connected to the \textsc{False} node.

When the low branch (dashed arc) is taken from the root, $v_1$ is excluded from the current branch. The algorithm then backtracks to the state before selecting $v_1$, so the partial stable set is set to $\mathcal{R}=\emptyset$, and proceeds to the next vertex in the ordering, namely $v_2$. Figures~\ref{fig2:sub-2}--\ref{fig2:sub-4} illustrate the subsequent construction steps for the branches associated with $S_2=\{v_2,v_4,v_6\}$ and $S_3=\{v_2,v_5\}$, which are obtained through different high/low decisions after excluding $v_1$. Along the branch generating $S_3$, the node to be created, denoted by $a'$, is equivalent to the existing node $a$. Hence, $a'$ is merged with $a$, reducing the size of the ZDD while preserving its representation. Continuing the recursive exploration over all possible branches produces the complete ZDD shown in Figure~\ref{fig2:sub-9}, which encodes the four maximal stable sets $S_1=\{v_1,v_3,v_5\}$, $S_2=\{v_2,v_4,v_6\}$, $S_3=\{v_2,v_5\}$, and $S_4=\{v_3,v_6\}$.

\begin{figure}[htbp]
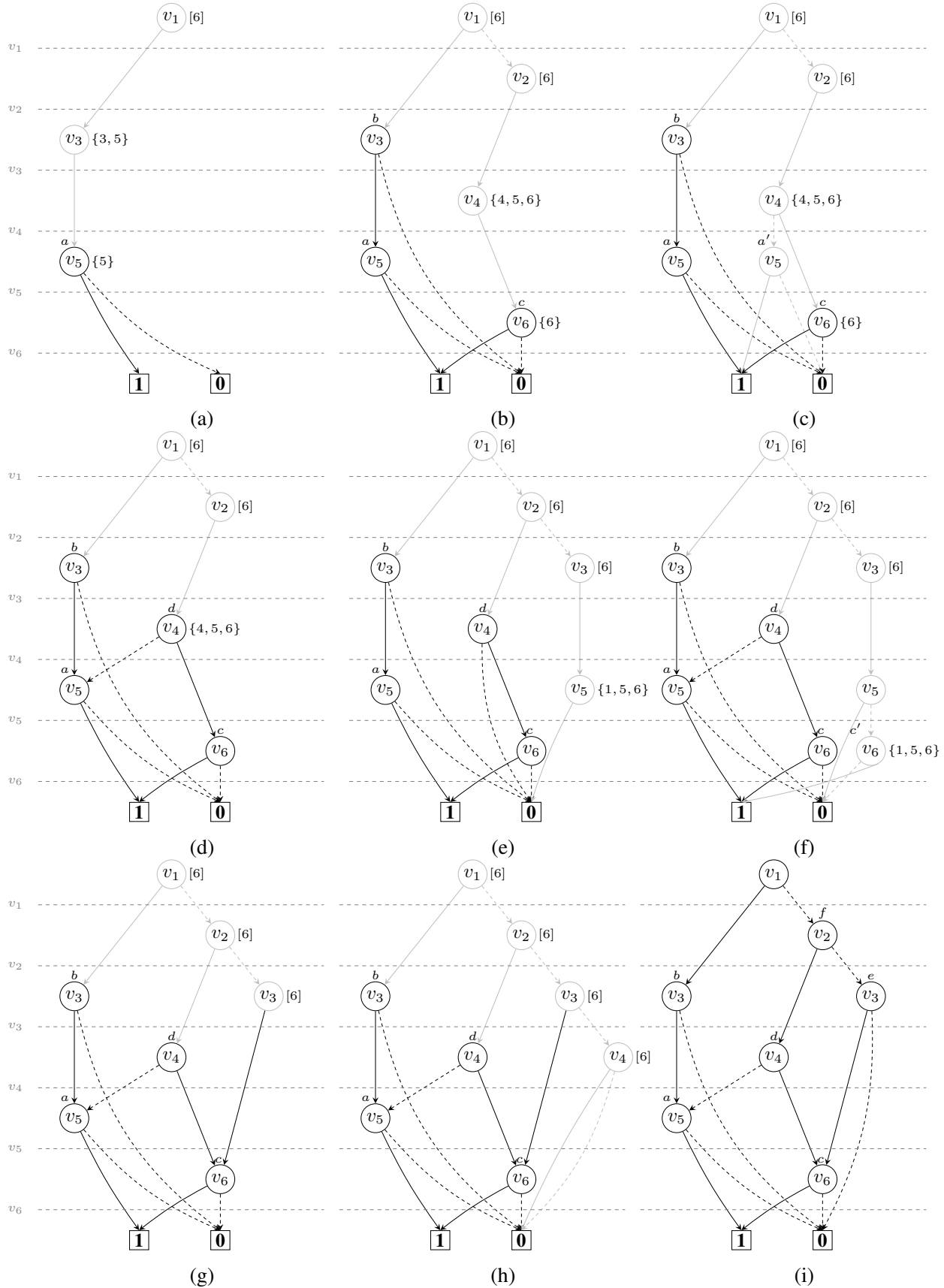

  \centering
  \tikzset{
    >=stealth,
    one arc/.style={->},
    zero arc/.style={->, dashed, dash pattern=on 1.8pt off 1.6pt}
  }
\begin{adjustbox}{max totalsize={\textwidth}{0.99\textheight},center}
    \begin{minipage}{\textwidth}
	\begin{subfigure}[t]{0.33\textwidth}
		\centering

		\caption{}
		\label{fig2:sub-1}
	\end{subfigure}\hfill
\begin{subfigure}[t]{0.33\textwidth}
	\centering
	%
	\caption{}
	\label{fig2:sub-2}
\end{subfigure}\hfill
\begin{subfigure}[t]{0.33\textwidth}
	\centering
	%
	\caption{}
	\label{fig2:sub-3}
\end{subfigure}\hfill
\begin{subfigure}[t]{0.33\textwidth}
	\centering
	%
		\caption{}
		\label{fig2:sub-9}
	\end{subfigure}
    \end{minipage}
  \end{adjustbox}
  \caption{Example of construction of the complete ZDD for the graph of Figure \ref{fig:comparison_zdd}.}
  \label{fig:six-subfigs}
\end{figure}

\begin{figure*}[htbp]
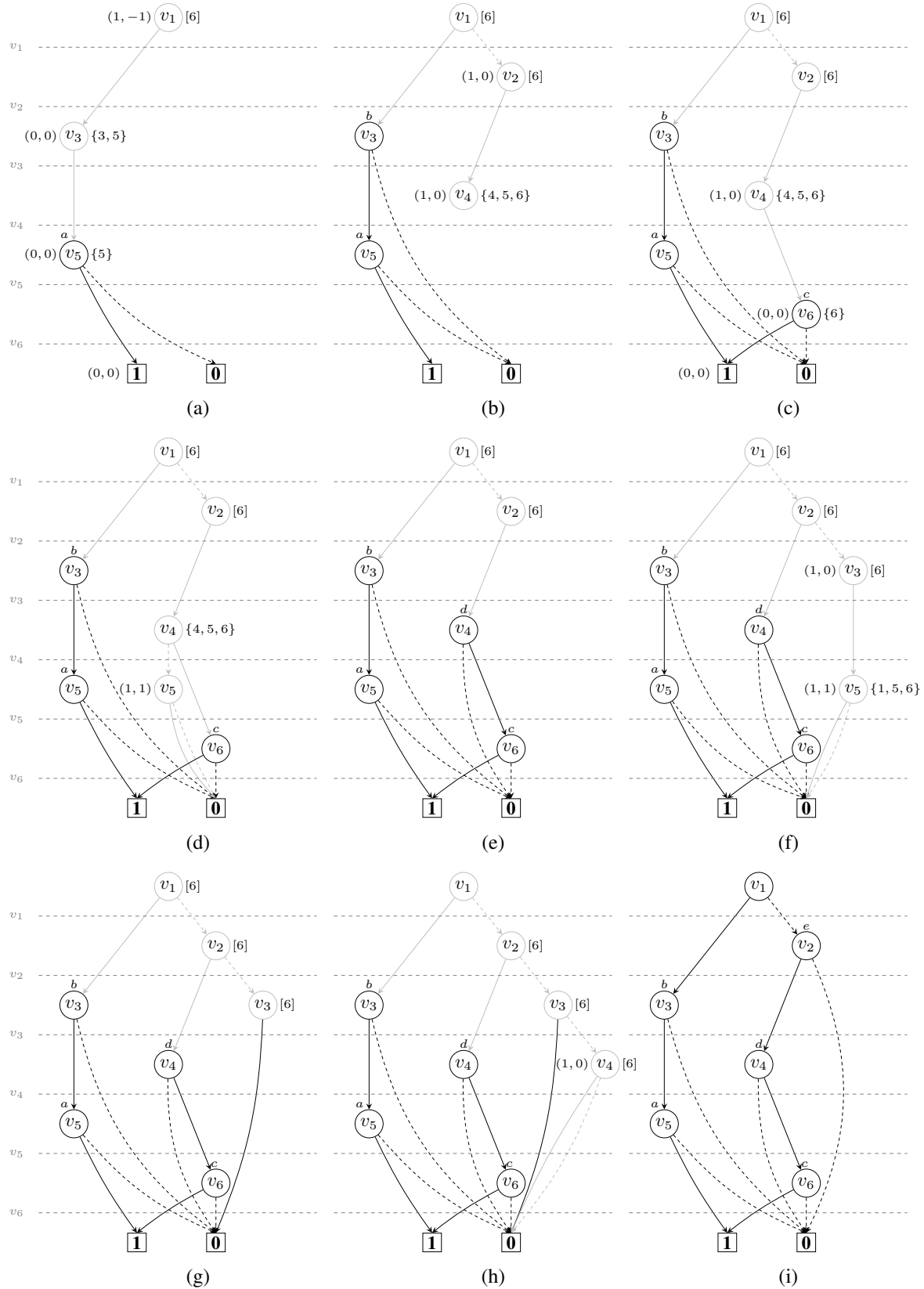

    \centering
    \tikzset{
        >=stealth,
        one arc/.style={->},
        zero arc/.style={->, dashed, dash pattern=on 1.8pt off 1.6pt}
    }
	\begin{adjustbox}{max totalsize={\textwidth}{0.95\textheight},center}
	\begin{minipage}{\textwidth}
	\centering\setlength{\parindent}{0pt}
	\begin{subfigure}[t]{0.325\textwidth}
		\centering

		\caption{}
		\label{fig:sub-1}
	\end{subfigure}\hfill%
	\begin{subfigure}[t]{0.325\textwidth}
		\centering
		%
		\caption{}
		\label{fig:sub-2}
	\end{subfigure}\hfill%
	\begin{subfigure}[t]{0.325\textwidth}
		\centering
		%
		\caption{}
		\label{fig:sub-3}
	\end{subfigure}\par\medskip
	\begin{subfigure}[t]{0.325\textwidth}
		\centering
		%
		\caption{}
		\label{fig:sub-4}
	\end{subfigure}\hfill%
	\begin{subfigure}[t]{0.325\textwidth}
		\centering
		%
	\caption{}
	\label{fig:sub-6}
	\end{subfigure}\par\medskip
	\begin{subfigure}[t]{0.325\textwidth}
		\centering
		%
        
	\caption{}
	\label{fig:sub-9}
\end{subfigure}
    \end{minipage}
  \end{adjustbox}
	\caption{Example of construction of the reduced ZDD for the graph of Figure \ref{fig:comparison_zdd}, using the feasible dual vector $\boldsymbol{\pi}=(1,0,0,1,0,0)$.}
	\label{fig:reducedZDD}
\end{figure*}

Figure~\ref{fig:reducedZDD} illustrates the key steps for constructing the reduced ZDD under \(UB=3\) and the feasible dual vector \(\boldsymbol{\pi}=(1,0,0,1,0,0)\). Although the graph in Figure~\ref{fig:comparison_zdd}(a) has chromatic number \(\chi(\mathcal{G})=2\), we set \(UB=3\) in this illustrative example so that the reduced ZDD is constructed with respect to the target value \(\tau = UB-1=2\). In this way, the reduction aims to preserve the maximal stable sets that may appear in an optimal coloring of cost \(2\), while pruning only those branches that cannot contribute to any solution that improves the incumbent upper bound.

We use the same notation as in Figure~\ref{fig:six-subfigs}: \([n]\) denotes \(\{1,2,\dots,n\}\), and the set in braces represents the uncovered set \(\mathcal{U}\) of the corresponding recursive call. In addition, the two values shown in parentheses \((x,y)\) have the following meaning. 
The first value \(x\) is the cumulative reduced cost of the current partial path, computed from the vertices selected by high-branch decisions up to the current node. The second value \(y\) is a valid lower bound on the reduced cost of any maximal stable set that can still be generated from the corresponding branch. Thus, if \(y\) exceeds the fixing threshold, the whole branch can be safely pruned. This lower bound is computed according to the pruning test described in Algorithm~\ref{makeIS} and Theorem~\ref{thm:pruning_rule}. The detailed explanation of each figure follows.

\begin{enumerate}[{\ref{fig:reducedZDD}}(a){:}]
    \item Along the $v_1$-associated branch, the reduced cost remains below $1$, hence the path is directed to the \textsc{True} node. When the search takes the root's high branch, the uncovered set is $U=\{v_3,v_5\}$, and Theorem~\ref{thm:pruning_rule} yields a branch lower bound of $0$, so no pruning is applied. Along the path $S_1=\{v_1,v_3,v_5\}$, upon reaching node $a$ and examining its high branch, we obtain the cumulative reduced cost of this path $RC_{hi(a)}=0$, which does not exceed the fixing threshold \(\delta=0\). Hence, $S_1$ is retained in the reduced ZDD.
    
    \item When \(v_2\) is selected, the pruning test does not apply, and the recursion proceeds to the next branching vertex, \(v_4\). 
    
    \item Selecting \(v_4\) then forces the selection of \(v_6\), since excluding \(v_6\) would violate maximality.

    \item If \(v_4\) is not selected, the reduced-cost lower bound of the corresponding branch triggers pruning, and the node associated with \(v_5\) is linked to the \textsc{False} node. Consequently, the path corresponding to \(S_3=\{v_2,v_5\}\) is not inserted into the reduced ZDD.
    

    \item By the ZDD property, the node associated with \(v_5\) is eliminated, and the low arc of the node associated with \(v_4\) is connected directly to the \textsc{False} node.

    \item On the branch associated with \(v_3\), selecting \(v_3\) yields a reduced-cost lower bound equal to \(1\), which triggers pruning. Therefore, this branch is rejected.

    \item Since the path associated with \(v_3\) satisfies the pruning condition, it is connected to the \textsc{False} node. Consequently, the path corresponding to \(S_4=\{v_3,v_6\}\) is pruned.

    \item When branching on \(v_4\), note that \(v_4\), \(v_5\), and \(v_6\) are all non-neighbors of \(v_2\). Hence, any stable set generated along this branch could still be augmented by \(v_2\), and therefore cannot be maximal. Consequently, the branch is rejected.

    \item  Nodes whose high arc is connected to the \textsc{False} node are eliminated. As a result, the node associated with \(v_2\) is also connected to the \textsc{False} node. All branches have been processed. The algorithm terminates, yielding the reduced ZDD shown in Figure~\ref{fig:sub-9}.
\end{enumerate}

To summarize, the complete ZDD in Figure~\ref{fig2:sub-9} contains 7 decision nodes and 2 terminal nodes, together with 14 arcs, for a total size of 23. It encodes all four maximal stable sets for the example graph in Figure~\ref{fig:comparison_zdd}(a):
\[
S_1=\{v_1,v_3,v_5\},\quad
S_2=\{v_2,v_4,v_6\},\quad
S_3=\{v_2,v_5\},\quad
S_4=\{v_3,v_6\}.
\]
In contrast, the reduced ZDD in Figure~\ref{fig:sub-9} contains six decision nodes and two terminal nodes, together with twelve arcs, for a total size of 20. It retains only two maximal stable sets:
\[
S_1=\{v_1,v_3,v_5\}
\quad\text{and}\quad
S_2=\{v_2,v_4,v_6\},
\]
whereas the paths corresponding to \(S_3\) and \(S_4\) are pruned by the reduced-cost fixing techniques. The two retained maximal stable sets exactly form an optimal coloring of the graph, since they cover all vertices using two colors.
From a computational perspective, the ZDD-based pricing procedure can be implemented by dynamic programming over the ZDD and has complexity linear in the size of the diagram, namely \(O(|\mathcal{N}|+|\mathcal{A}|)\), where \(\mathcal{N}\) and \(\mathcal{A}\) denote the sets of ZDD nodes and arcs, respectively. In this example, the complete ZDD requires \(O(9+14)=O(23)\), whereas the reduced ZDD requires \(O(8+12)=O(20)\). This small example, therefore, illustrates how reduced-cost fixing decreases the size of the pricing structure while preserving the maximal stable sets required for optimality.

\section{Additional Details on the Computational Results}
\label{sec:additional_results}

This section presents additional details about the computational evaluation of \BPname. We first report details on the classical DIMACS benchmark instances (\S\ref{sec:DIMACS}). We then provide a detailed comparison with state-of-the-art exact algorithms for which source code is available, highlighting the relative performance of \BPname\ (\S\ref{subsec:detailed_comparison_with_code}). Finally, we analyze the effectiveness of the algorithm's main components to assess their individual contributions to overall performance (\S\ref{subsec:component_effectiveness}).

\subsection{DIMACS Instances}\label{sec:DIMACS}

The DIMACS instances can be divided into 22 different families. For example, the \texttt{DSJC} class consists of random graphs with varying densities, which pose severe challenges for exact algorithms because of their extremely high edge density. The \texttt{myciel} class is constructed with special structures. Although these graphs have relatively few edges, they hide high chromatic numbers, making them particularly challenging. Other categories include structured graphs, such as the \texttt{queen} instances, and real-world problem-based instances, such as the \texttt{fpsol2} class, each presenting unique computational difficulties. The number of instances per family ranges from 2, for the \texttt{school} family, to 14, for the \texttt{FullIns} family, with an average of approximately 6.2. Three instances, \texttt{games120}, \texttt{latin\_square\_10}, and \texttt{will199GPIA}, which do not belong to any larger group, are collected under the family labeled \texttt{others}. Instances with proper names, namely \texttt{anna}, \texttt{david}, \texttt{homer}, \texttt{huck}, and \texttt{jean}, are grouped into the family labeled \texttt{SocialNet}, as they are derived from social settings.

The computational results presented in~\cite{van2022graph} summarize the state-of-the-art upper and lower bounds for the DIMACS instances available at the time. In this work, these values are further updated by incorporating results from more recent exact and heuristic algorithms for the GCP. The resulting bounds, therefore, reflect the best-known values currently available from the literature. Based on these results, Table~\ref{tab:families} reports, for each instance family, the number of instances solved to optimality in the literature (column ``\#solved''). These are instances for which the current lower and upper bounds on the chromatic number $\chi(\mathcal{G})$ coincide.
It is worth noting that the best-known lower and upper bounds often stem from different algorithmic approaches. Moreover, neither the computing times required to obtain them nor the specific algorithms used are reported in~\cite{van2022graph}.  
As a result, it is difficult to directly compare the performance of any new exact or heuristic algorithm against these results.  
Nevertheless, we include them to indicate which instances have been solved to optimality in the literature and to provide a broader context for evaluating the performance of \BPname\ (see \cite{van2022graph} for further details).

As shown in Table~\ref{tab:families}, several families contain only solved instances. These include \texttt{DSJR}, \texttt{FullIns}, \texttt{GPIA}, \texttt{SocialNet}, \texttt{fpsol}, \texttt{inithx}, \texttt{le}, \texttt{miles}, \texttt{mug}, \texttt{mulsol}, \texttt{myciel}, \texttt{qg.order}, \texttt{school}, and \texttt{zeroin}, where all instances have been solved to optimality. These families generally consist of instances with moderate size and structure that are well-suited to existing exact algorithms.
On the other hand, some families remain particularly challenging because they contain a significant number of unresolved instances. In particular, the \texttt{C} family remains completely unsolved. Other difficult families include \texttt{DSJC}, with only 4 out of 12 instances solved; \texttt{Insertions} (9 out of 11); \texttt{flat} (5 out of 6); \texttt{queen} (12 out of 13); \texttt{r} (8 out of 9); \texttt{wap} (5 out of 8); and \texttt{others} (2 out of 3). These families tend to include larger and denser graphs, which make them more challenging for current exact methods.

The last six columns of Table~\ref{tab:families} report, for each family, the minimum and maximum number of vertices $|\mathcal{V}|$, number of edges $|\mathcal{E}|$, and edge density values $\mu$, expressed as percentages. 
The smallest instance has 11 vertices, while the largest has 10{,}000; the number of edges ranges from 20 to over four million. The edge densities vary considerably across the benchmark: some families, such as \texttt{FullIns}, \texttt{GPIA}, \texttt{Insertions}, and \texttt{SocialNet}, include very sparse instances with minimum densities around or below 1\%, while others, like \texttt{C}, \texttt{DSJC}, \texttt{DSJR}, and \texttt{r}, contain highly connected graphs with densities of 90\% or higher. This wide range of densities illustrates the structural diversity of the DIMACS instances and highlights the varying difficulty these instances pose for exact GCP algorithms.

\begin{table}[] 
  \centering
    \footnotesize
  \caption{Main features of the 22 families in the 137 DIMACS instances.}
   \scalebox{1.0}{
    \begin{tabular}{lrrrrrrrrrrr}
    \toprule
&&&& \multicolumn{2}{c}{$|\mathcal{V}|$} && \multicolumn{2}{c}{$|\mathcal{E}|$} && \multicolumn{2}{c}{$\mu$ ($\%$)}\\
 \cmidrule(r){5-6} \cmidrule(r){8-9} \cmidrule(r){11-12}
    Family & \#inst & \#solved && $\min$ & $\max$ && $\min$ & $\max$ && $\min$  & $\max$  \\
\cmidrule(r){1-3} \cmidrule(r){5-6} \cmidrule(r){8-9} \cmidrule(r){11-12}
        \texttt{C} 	&	3	&	0	&	&	2,000	&	4,000	&	&	999,836	&	4,000,268	&	&	50.0	&	90.0	\\
    \texttt{DSJC} 	&	12	&	4	&	&	125	&	1,000	&	&	736	&	449,449	&	&	9.5	&	90.1	\\
    \texttt{DSJR} 	&	3	&	3	&	&	500	&	500	&	&	3,555	&	121,275	&	&	2.8	&	97.2	\\
    \texttt{FullIns} 	&	14	&	14	&	&	30	&	4,146	&	&	100	&	77,305	&	&	0.9	&	23.0	\\
    \texttt{GPIA} 	&	4	&	4	&	&	662	&	1,916	&	&	4,185	&	65,390	&	&	0.7	&	5.4	\\
    \texttt{Insertions} 	&	11	&	9	&	&	37	&	1,406	&	&	72	&	9,695	&	&	1.0	&	10.8	\\
    \texttt{SocialNet} 	&	5	&	5	&	&	74	&	561	&	&	254	&	1,628	&	&	1.0	&	11.1	\\
    \texttt{flat} 	&	6	&	5	&	&	300	&	1,000	&	&	21,375	&	246,708	&	&	47.7	&	49.4	\\
    \texttt{fpsol} 	&	3	&	3	&	&	425	&	496	&	&	8,688	&	11,654	&	&	8.6	&	9.6	\\
    \texttt{inithx} 	&	3	&	3	&	&	621	&	864	&	&	13,969	&	18,707	&	&	5.0	&	7.3	\\
    \texttt{le} 	&	12	&	12	&	&	450	&	450	&	&	5,714	&	17,425	&	&	5.7	&	17.2	\\
    \texttt{miles} 	&	5	&	5	&	&	128	&	128	&	&	387	&	5,198	&	&	4.8	&	64.0	\\
    \texttt{mug} 	&	4	&	4	&	&	88	&	100	&	&	146	&	166	&	&	3.4	&	3.8	\\
    \texttt{mulsol} 	&	5	&	5	&	&	184	&	197	&	&	3,885	&	3,973	&	&	20.3	&	23.3	\\
    \texttt{myciel} 	&	5	&	5	&	&	11	&	191	&	&	20	&	2,360	&	&	13.0	&	36.4	\\
    \texttt{others} 	&	3	&	2	&	&	120	&	900	&	&	638	&	307,350	&	&	2.9	&	76.0	\\
    \texttt{qg.order} 	&	4	&	4	&	&	900	&	10,000	&	&	26,100	&	990,000	&	&	2.0	&	6.5	\\
    \texttt{queen} 	&	13	&	12	&	&	25	&	256	&	&	160	&	6,320	&	&	19.4	&	53.3	\\
    \texttt{r} 	&	9	&	8	&	&	125	&	1,000	&	&	209	&	485,090	&	&	2.7	&	97.1	\\
    \texttt{school} 	&	2	&	2	&	&	352	&	385	&	&	14,612	&	19,095	&	&	23.7	&	25.8	\\
    \texttt{wap} 	&	8	&	5	&	&	905	&	5,231	&	&	43,081	&	294,902	&	&	2.2	&	10.5	\\
    \texttt{zeroin} 	&	3	&	3	&	&	206	&	211	&	&	3,540	&	4,100	&	&	16.0	&	18.5	\\
     
    \midrule
    total/$\min$/$\max$ & 137   & 117   && 11    & 10,000 && 20     & 4,000,268 && 0.7 & 97.2   \\
    \bottomrule
    \end{tabular}
    }
  \label{tab:families}
\end{table}

\subsection{Detailed Comparative Results with Exact Algorithms with Available Source Code} \label{subsec:detailed_comparison_with_code}

Table~\ref{tab:tableForInstance} reports detailed comparative results between \BPname\ and the exact algorithms with publicly available source code on the full set of 137 DIMACS instances. These include the BP algorithm of~\citet{Held2012}, called \texttt{exactcolors}, and the DSatur-based algorithm of~\citet{san2012new}, called \texttt{DSATUR}. The remaining five are SAT-based algorithms: \texttt{gc-cdcl} by~\citet{hebrard2020constraint}, \texttt{CliColCom} by~\citet{heule2022cliques}, \texttt{POP-S} by~\citet{faber2024sat}, \texttt{Assignment} and \texttt{ZykovColor} by~\citet{brand2026customized}. The first six columns present the instance name (Instance), the family of the instance (Family), the number of vertices ($|\mathcal{V}|$), the number of edges ($|\mathcal{E}|$), the best known lower bound ($LB$), and the best known upper bound ($UB$) on the chromatic number. These bounds are initialized from the values reported by \citet{van2022graph} and subsequently updated with more recent results from the literature and improvements from our computational experiments. The lower bound updated by \BPname\ is highlighted in bold, for instance \texttt{C2000.9}. We improved the lower bound of instance \texttt{C2000.9} and successfully closed the previously open instance \texttt{r1000.1c} by proving its optimality. The computing time of the fastest algorithm for each instance is also shown in bold.
Columns 7--15 report the running times of \BPname, \BPUB, and seven reference algorithms, all evaluated with a time limit of 3600 seconds. To ensure a fair comparison, all algorithms with publicly available source code were executed on the same machine under the computational environment described in Section~\ref{sec:compenv}.

\begin{landscape}
\scriptsize
\setlength{\tabcolsep}{2pt}
\renewcommand{\arraystretch}{0.93}
\begin{longtable}{l l r r r r *{9}{r}}
\caption{Comparison of the exact algorithms for the GCP on the 137 DIMACS instances, with a time limit of 3600 seconds. Computing times below 0.01 seconds are reported as $\le 0.01$ seconds. } \label{tab:tableForInstance} \\
\bottomrule
\multicolumn{1}{l}{{\tt Instance}} &
\multicolumn{1}{l}{{\tt Family}} &
\multicolumn{1}{c}{{\tt $|\mathcal{V}|$}} &
\multicolumn{1}{c}{{\tt $|\mathcal{E}|$}} &
\multicolumn{1}{c}{{\tt $LB$}} &
\multicolumn{1}{c}{{\tt $UB$}} &
\multicolumn{1}{c}{{\tt exactcolors}} &
\multicolumn{1}{c}{{\tt DSATUR}} &
\multicolumn{1}{c}{{\tt gc-cdcl}} &
\multicolumn{1}{c}{{\tt CliColCom}} &
\multicolumn{1}{c}{{\tt POP-S}} &
\multicolumn{1}{c}{{\tt Assignment}} &
\multicolumn{1}{c}{{\tt ZykovColor}} &
\multicolumn{1}{c}{\BPname} &
\multicolumn{1}{c}{\BPUB} \\

\midrule
\endfirsthead

\multicolumn{15}{l}{Continued}\\
\bottomrule
\multicolumn{1}{l}{{\tt Instance}} &
\multicolumn{1}{l}{{\tt Family}} &
\multicolumn{1}{c}{{\tt $|\mathcal{V}|$}} &
\multicolumn{1}{c}{{\tt $|\mathcal{E}|$}} &
\multicolumn{1}{c}{{\tt $LB$}} &
\multicolumn{1}{c}{{\tt $UB$}} &
\multicolumn{1}{c}{{\tt exactcolors}} &
\multicolumn{1}{c}{{\tt DSATUR}} &
\multicolumn{1}{c}{{\tt gc-cdcl}} &
\multicolumn{1}{c}{{\tt CliColCom}} &
\multicolumn{1}{c}{{\tt POP-S}} &
\multicolumn{1}{c}{{\tt Assignment}} &
\multicolumn{1}{c}{{\tt ZykovColor}} &
\multicolumn{1}{c}{\BPname} &
\multicolumn{1}{c}{\BPUB} \\
\midrule
\endhead

      \bottomrule
      \multicolumn{15}{l}{Continued on next page}\\
      \endfoot

      \toprule
      \endlastfoot
\texttt{1-FullIns\_3} 	&	 FullIns 	&	 30 	&	 100 	&	 4 	&	 4 	&	 $\le \textbf{0.01}$ 	&	 $\le \textbf{0.01}$ 	&	 0.02 	&	 0.03 	&	 0.187 	&	 0.06 	&	 0.11 	&	 $\le \textbf{0.01}$ 	&	 $\le 0.01$ \\
\texttt{1-FullIns\_4} 	&	 FullIns 	&	 93 	&	 593 	&	 5 	&	 5 	&	 t.l. 	&	 $\le \textbf{0.01}$ 	&	 0.02 	&	 0.03 	&	 0.184 	&	 0.06 	&	 0.06 	&	 $\le \textbf{0.01}$ 	&	 $\le 0.01$ \\
\texttt{1-FullIns\_5} 	&	 FullIns 	&	 282 	&	 3247 	&	 6 	&	 6 	&	 t.l. 	&	 $\le \textbf{0.01}$ 	&	 0.11 	&	 0.17 	&	 0.714 	&	 0.06 	&	 0.06 	&	 $\le \textbf{0.01}$ 	&	 $\le 0.01$ \\
\texttt{1-Insertions\_4} 	&	 Insertions 	&	 67 	&	 232 	&	 5 	&	 5 	&	 t.l. 	&	 t.l. 	&	 t.l. 	&	 1.28 	&	 0.916 	&	 \textbf{0.37} 	&	 136.82 	&	 t.l. 	&	 t.l. \\
\texttt{1-Insertions\_5} 	&	 Insertions 	&	 202 	&	 1227 	&	 6 	&	 6 	&	 t.l. 	&	 t.l. 	&	 t.l. 	&	 t.l. 	&	 t.l. 	&	 t.l. 	&	 t.l. 	&	 t.l. 	&	 t.l. \\
\texttt{1-Insertions\_6} 	&	 Insertions 	&	 607 	&	 6337 	&	 4 	&	 7 	&	 t.l. 	&	 t.l. 	&	 t.l. 	&	 t.l. 	&	 t.l. 	&	 t.l. 	&	 t.l. 	&	 t.l. 	&	 t.l. \\
\texttt{2-FullIns\_3} 	&	 FullIns 	&	 52 	&	 201 	&	 5 	&	 5 	&	 $\le \textbf{0.01}$ 	&	 $\le \textbf{0.01}$ 	&	 $\le \textbf{0.01}$ 	&	 0.03 	&	 0.169 	&	 0.06 	&	 0.06 	&	 0.03 	&	 0.02 \\
\texttt{2-FullIns\_4} 	&	 FullIns 	&	 212 	&	 1621 	&	 6 	&	 6 	&	 t.l. 	&	 $\le \textbf{0.01}$ 	&	 0.08 	&	 0.05 	&	 0.358 	&	 0.12 	&	 0.13 	&	 3.25 	&	 0.26 \\
\end{longtable}
\end{landscape}

Running times below or equal to 0.01 seconds are reported as ``$\le0.01$'' to avoid comparisons of very small values and because our focus is on evaluating performance on the harder instances.
If an algorithm fails to solve an instance within the time limit, the entry is marked as ``t.l.'' (time limit).

The results reported in Table~\ref{tab:tableForInstance} demonstrate the competitiveness of our approach. Overall, \BPname\ solves 96 instances, outperforming all competing algorithms with publicly available source code. In terms of computational efficiency, \BPname\ achieves the fastest running time on 73 instances among the seven reference algorithms considered in the comparison.
When compared with \texttt{exactcolors}, the only other BP algorithm in this set, \BPname\ demonstrates markedly superior performance: it solves 31 additional instances. Among the 65 instances solved by \texttt{exactcolors}, our method achieves faster runtimes on 45 cases, while matching its performance on the remaining 5. 

These findings confirm that \BPname\ provides substantial runtime advantages, even against algorithms employing the same exact framework.

\subsection{Effectiveness of the Main Components of \BPname} \label{subsec:component_effectiveness}

This section presents an analysis of the effectiveness of the main \BPname\ components (\S\ref{subsec:efficiency}) and of the ZDD reduction strategies (\S\ref{subsubsec:effect_zdd_reduction}).

\subsubsection{Key Efficiency Drivers of \BPname.}\label{subsec:efficiency}
To assess the contribution of each component, we evaluate the following \BPname\ variants on the DIMACS instances, each obtained by modifying or disabling a single algorithmic component:
\begin{itemize}
    \item $\texttt{RootZDDPric}$: ZDD-based pricing at the root node only (\S\ref{subsec:lb_zdd}).
    \item $\texttt{NoHeurPric}$: Removes the heuristic pricing algorithm at the root node (\S\ref{subsec:lb_zdd}).
    \item $\texttt{CompleteZDD}$: Disables all reduced-cost fixing strategies (\S\ref{sec:prica-ZDD-enh}).
\end{itemize}

\begin{figure}[t]
\caption{Number of instances solved as a function of computing time: comparison of three variant algorithms with \BPname\ for the GCP.}
\label{fig:survivalPlotVariant}
\center
\includegraphics[width=0.95\columnwidth]{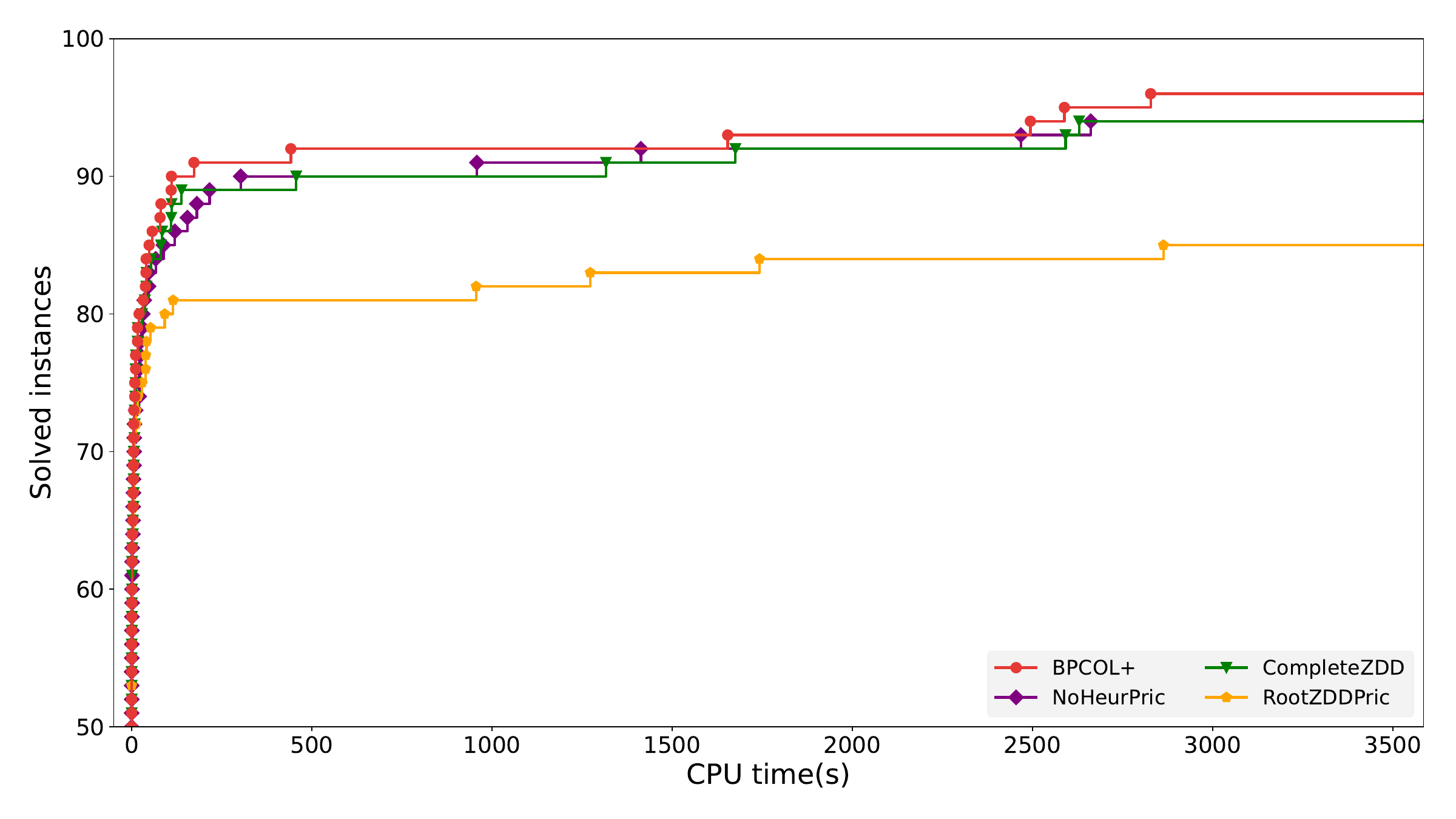}
\end{figure}

Figure~\ref{fig:survivalPlotVariant} reports the number of instances solved as a function of computing time for \BPname\ and its three variants on the 137 DIMACS benchmark instances, where the x-axis denotes CPU time, and the y-axis denotes the cumulative number of solved instances. Overall, \BPname\ achieves the best performance, solving 96 instances within the 1-hour time limit, compared with 94 for \texttt{NoHeurPric}, 94 for \texttt{CompleteZDD}, and 85 for \texttt{RootZDDPric}. Its curve stays above the others for most of the time horizon, showing that the full combination of algorithmic components improves both robustness and time-to-solution.
The relatively poor performance of \texttt{RootZDDPric} indicates that replacing the root-node pricing procedure with the ZDD-based dynamic programming approach substantially weakens the algorithm, as constructing and processing complete ZDDs can be difficult for large or sparse instances. The smaller but consistent gaps between \BPname\ and \texttt{NoHeurPric} or \texttt{CompleteZDD} further demonstrate the value of the heuristic pricing mechanism and reduced-cost fixing strategies. Although these two variants remain competitive on easier instances, they eventually fall behind \BPname, confirming that both components contribute to faster convergence and stronger overall solving capability.

\subsubsection{Effectiveness of ZDD Reduction Strategies.}\label{subsubsec:effect_zdd_reduction}

To evaluate the effectiveness of the proposed ZDD reduction mechanism, we compare several reduction strategies derived from different fixing rules and dual formulations. The comparison is conducted on a selected subset of DIMACS instances for which the effect of eliminating maximal stable sets can be explicitly assessed. We retain instances satisfying the following criteria: (i) the LP optimum at the root node is non-integral; and (ii) the gap between the initial upper and lower bounds does not exceed~2. This selection yields 18 instances in total.

The baseline, denoted by ``Complete'', corresponds to the complete ZDD containing all maximal stable sets. The remaining strategies construct reduced ZDDs by using dual vectors generated either from Lemma~\ref{lemma:sc_fix} or from dual-formulation-based variants consistent with the framework described in Section~\ref{sec:genaltdual}. Specifically, ``Lemma~\ref{lemma:sc_fix}'' reports the number of maximal stable sets retained in the reduced ZDD when only Lemma~\ref{lemma:sc_fix} is used. The strategies ``\emph{DF\textsubscript{Y}}'' and ``\emph{DF\textsubscript{L}}'' generate dual vectors using two alternative formulations adapted from \citet{yang2024} and \citet{De2023}, respectively. Both variants follow the same dual-vector generation framework described in Section~\ref{sec:genaltdual}. Finally, ``\emph{DF}'' denotes the proposed dual-formulation-based approach of Section~\ref{sec:genaltdual}, which combines the reduced-cost-based redundancy test with Lemma~\ref{lemma:infeasible_fix}.

Table~\ref{tableForReductionVariants} reports the remaining number of maximal stable sets (columns ``\#mss'') and, for dual-formulation-based strategies, the time required to generate the corresponding dual vectors (columns ``time''). For the columns \#mss, the smallest value in each row is highlighted in bold. 

The results show that the proposed \emph{DF} strategy is consistently the strongest, producing the fewest remaining maximal stable sets or tying for the best result. This indicates that the proposed dual formulation is at least as effective as the competing formulations on all tested instances, and strictly improves the reduction on several cases.

\begin{table}[t]
  \centering
  \caption{Comparison of different ZDD reduction strategies over 18 selected DIMACS instances}
  \label{tableForReductionVariants}
  \footnotesize
  \renewcommand{\arraystretch}{1.35}
  \setlength{\tabcolsep}{4pt}

  \begin{tabular}{>{\ttfamily}lrrrrrrrr}
  \toprule
  \multicolumn{1}{l}{\multirow{2}{*}{\normalfont Instance}}
  & \multicolumn{1}{c}{Complete}
  & \multicolumn{1}{c}{Lemma~\ref{lemma:sc_fix}}
  & \multicolumn{2}{c}{$DF_Y$}
  & \multicolumn{2}{c}{$DF_L$}
  & \multicolumn{2}{c}{DF} \\
  \cmidrule(lr){2-2} \cmidrule(lr){3-3} \cmidrule(lr){4-5}
  \cmidrule(lr){6-7} \cmidrule(lr){8-9}
  & \multicolumn{1}{c}{\#mss}
  & \multicolumn{1}{c}{\#mss}
  & \multicolumn{1}{c}{\#mss} & \multicolumn{1}{c}{time}
  & \multicolumn{1}{c}{\#mss} & \multicolumn{1}{c}{time}
  & \multicolumn{1}{c}{\#mss} & \multicolumn{1}{c}{time} \\
  \midrule

    2-FullIns\_4      & 218       & 218       & 112       & 0.27   & 112       & 0.25   & \textbf{85}     & 0.27 \\
    2-Insertions\_3   & 3161      & 3161      & 2679      & 0.55   & 2679      & 0.53   & \textbf{2632}   & 0.69 \\
    3-FullIns\_4      & 953       & 953       & 575       & 0.40   & 575       & 0.39   & \textbf{540}    & 0.37 \\
    3-Insertions\_3   & 228439    & 228439    & 219406    & 0.81   & 219406    & 0.81   & \textbf{218732} & 1.60 \\
    4-FullIns\_4      & 138       & 138       & \textbf{79} & 0.42  & \textbf{79} & 0.43  & \textbf{79}     & 0.41 \\
    4-Insertions\_3   & 37833929  & 37833929  & 37302681  & 1.34   & 37302681  & 1.35   & \textbf{37274394} & 4.40 \\
    5-FullIns\_4      & 182       & 182       & 106       & 0.51   & 106       & 0.48   & \textbf{92}     & 0.48 \\
    DSJC125.5         & 43268     & 22609     & 22318     & 25.26  & 22325     & 26.45  & \textbf{22294}  & 25.60 \\
    DSJC125.9         & 524       & 392       & 360       & 4.74   & 361       & 4.70   & \textbf{356}    & 4.53 \\
    DSJC250.9         & 2580      & 2567      & 2556      & 35.16  & \textbf{2555} & 35.44 & \textbf{2555}   & 35.82 \\
    DSJR500.1c        & 385       & 383       & \textbf{377} & 4.20  & 379       & 4.18   & \textbf{377}    & 4.09 \\
    queen10\_10       & 376692    & 1540      & 579       & 37.68  & 707       & 39.01  & \textbf{550}    & 36.13 \\
    queen11\_11       & 2640422   & 1820566   & 21857     & 32.39  & 4043      & 31.08  & \textbf{3739}   & 33.21 \\
    queen12\_12       & 19469324  & 13602600  & 32295     & 44.35  & 22088     & 43.66  & \textbf{9440}   & 50.91 \\
    queen13\_13       & 151978440 & 106620568 & \textbf{52008} & 73.39 & \textbf{52008} & 67.20 & \textbf{52008}  & 75.94 \\
    queen14\_14       & -         & -         & 385384    & 117.25 & 582384    & 123.61 & \textbf{238088} & 122.41 \\
    queen15\_15       & -         & -         & 3177110   & 194.35 & \textbf{1484400} & 190.77 & \textbf{1484400} & 197.97 \\
    queen9\_9         & 57600     & 1496      & 751       & 13.00  & 1045      & 13.57  & \textbf{93}     & 12.04 \\
    \bottomrule
  \end{tabular}
\end{table}

The benefit of dual-based reduction is particularly clear when compared with the complete ZDD. For example, on \texttt{queen11\_11}, the complete ZDD represents 2,640,422 maximal stable sets, whereas the proposed \emph{DF} strategy reduces this number to 3,739. Similarly, on \texttt{queen12\_12}, the number of represented maximal stable sets decreases from 19,469,324 to 9,440. On the large queen instances, \texttt{queen14\_14} and \texttt{queen15\_15}, the complete ZDD cannot be constructed within the one-hour time limit, whereas the reduced variants remain constructible, with \emph{DF} retaining only 238,088 and 1,484,400 maximal stable sets, respectively.

Applying only Lemma~\ref{lemma:sc_fix} is not sufficient in general. On several instances, such as \texttt{2-FullIns\_4}, \texttt{2-Insertions\_3}, \texttt{3-FullIns\_4}, and \texttt{4-Insertions\_3}, it leaves exactly the same number of maximal stable sets as the complete ZDD. Although it can be effective on some structured instances, especially in the \texttt{queen} family, its reduction power is much weaker than that of the dual-formulation-based strategies. This confirms the need for a stronger fixing mechanism based on additional dual information.

Compared with the formulations inspired by \citet{yang2024} and \citet{De2023}, the proposed \emph{DF} formulation achieves the best or tied-best result on every tested instance. Although the two alternative formulations already remove a substantial number of maximal stable sets, \emph{DF} consistently matches or further strengthens their reduction effect over the entire selected test set.

The computational overhead of the proposed reduction remains moderate. For small- and medium-sized instances, generating dual vectors typically takes less than a few seconds. For larger instances, especially in the \texttt{queen} family, this time increases but remains acceptable relative to the substantial reduction in the number of maximal stable sets represented in the ZDD. Overall, Table~\ref{tableForReductionVariants} confirms that the proposed \emph{DF} strategy provides the most robust reduction among the tested variants, significantly decreasing the size of the represented set of maximal stable sets while keeping the additional computational cost manageable.

\end{document}